\documentclass[12pt]{amsart}

\usepackage{graphicx,indentfirst}
\usepackage{amsmath,amssymb,mathrsfs}
\usepackage{amsthm,amscd}
\usepackage{verbatim}
\usepackage{appendix}
\usepackage{enumitem,titletoc}
\usepackage{appendix}
\usepackage{imakeidx}
\makeindex[columns=2, title=Alphabetical Index]
\usepackage{fancyhdr}
\usepackage{amsfonts,color}
\usepackage[all]{xy}
\usepackage{tikz-cd}
\usepackage{syntonly}
\usepackage{fancyhdr}
\newtheorem{theorem}{Theorem}[section]
\newtheorem{lemma}{Lemma}[section]
\newtheorem{definition}{Definition}[section]
\newtheorem{remark}{Remark}[section]
\usepackage[left=2.55cm, right=2.55cm, bottom=3cm, top=3cm]{geometry}

\newcommand{\abs}[1]{|#1|^2}

\usepackage{tikz}
\usepackage{extarrows}
\usepackage{hyperref}

\def\XXint#1#2#3{{\setbox0=\hbox{$#1{#2#3}{\int}$ }
\vcenter{\hbox{$#2#3$ }}\kern-.6\wd0}}

\newtheorem{prop}{Proposition}[section]
\newtheorem{defn}{Definition}[section]
\newtheorem{corollary}{Corollary}[section]

\allowdisplaybreaks

\newcommand{\ddbar}{i\partial\bar\partial}
\newcommand{\ric}{\mathrm{Ric}}
\newcommand{\scR}{\mathrm{R}}
\newcommand{\innpro}[1]{\langle#1\rangle}
\newcommand{\bk}[1]{\Big(#1\Big)}

\newcommand{\xk}[1]{\big(#1\big)}

\newcommand{\cS}{\mathcal{S}}
\newcommand{\cR}{\mathcal{R}}

\makeindex

\DeclareMathOperator{\vol}{Vol}

\DeclareMathOperator{\tr}{tr}

\numberwithin{equation}{section}

\hypersetup{
    colorlinks=true,
    citecolor=green,
    filecolor=green,
    linkcolor=blue,
    urlcolor=black
}

\begin{document}

\address{Department of Mathematics \&~ Computer Science, Rutgers University, Newark, NJ 07102}

\email{bguo@rutgers.edu}

\address{Department of Mathematics, Rutgers University, Piscataway, NJ 08854}
\email{jiansong@math.rutgers.edu}

\title[Regularization of singular K\"ahler metrics]{Nash entropy, Calabi energy and geometric regularization of singular K\"ahler metrics}

\author{Bin Guo \and Jian Song}\date{}
\thanks{Work supported in part by the National Science Foundation under grants DMS-2203607 and DMS-2303508,  and the grant MPS-TSM-00946730 from Simons Foundation.}

\begin{abstract}  We prove uniform Sobolev bounds for solutions of the Laplace equation on a  general family of K\"ahler manifolds with bounded Nash entropy and Calabi energy.  These estimates establish a connection to the theory of RCD spaces and provide abundant examples of RCD spaces  topologically and holomorphically equivalent to projective varieties. Suppose $X$ is a normal projective variety that admits a resolution of singularities with relative nef or relative effective anti-canonical bundle. Then every admissible singular K\"ahler metric on $X$ with Ricci curvature bounded below induces a non-collapsed RCD space  homeomorphic to the projective variety $X$ itself.

\end{abstract}

\maketitle

\section{Introduction}

In the recent papers \cite{GPSS1, GPSS2} and their extensions \cite{GPSS3, Vu, GT}, various analytic and geometric estimates are established to build a general theory of geometric analysis on singular complex spaces. These estimates are built on the assumption on uniform bounds of the Nash entropy. In this paper, we aim to obtain new Sobolev estimates by imposing an additional assumption on the Calabi energy.  Such an assumption is satisfied in many geometric and analytic settings including the cases of constant scalar curvature K\"ahler metrics and K\"ahler metrics of bounded $L^2$-curvature. Our new estimates also reveal connections between singular K\"ahler spaces and the theory of RCD spaces.

Let $X$ be an $n$-dimensional compact normal K\"ahler variety equppied with a smooth K\"ahler metric  $\theta_X$ (c.f. Definition \ref{smm}). 
Similar to the singular K\"ahler metrics introduced in \cite{GPSS2}, we consider the set 
\begin{equation}\label{vspace}
\mathcal{V} (X, \theta_X, n, A,  p, K)
\end{equation}
 of all singular K\"ahler metrics $\omega$ satisfying the following properties. 
\begin{enumerate}

\item $\omega\in C^2(X^\circ)$, where $X^\circ=\cR(X)$ is the smooth part of $X$.

\medskip

\item $[\omega]$ is a K\"ahler class on $X$ with  
\begin{equation}\label{cls0}
I_\omega = [\omega]\cdot [\theta_X]^{n-1} \leq A.
\end{equation}
 
\medskip

\item  $p>n$ and 
\begin{equation}\label{nashen0}
\mathcal{N}_{\theta_X, p}(\omega) = \frac{1}{V_\omega}\int_X \left| \log\left( V_\omega^{-1} \frac{\omega^n}{\theta_X^n} \right) \right|^p\omega^n\leq K.
\end{equation}
where $V_\omega = [\omega]^n$ is the volume of $(X, \omega)$.

\medskip

\end{enumerate}
The singular K\"ahler metric $\omega$ is a closed positive $(1,1)$-current on $X$ with bounded local potentials due to the entropy bound (\ref{nashen0}) (c.f. \cite{Ko1, EGZ, Zz, GPT}). 
As defined in \cite{GPSS2}, we let 
 \begin{equation}\label{mmsp}
 (\hat X, d_\omega, \omega^n)
 \end{equation}
  be the metric completion of $(X^\circ, \omega|_{X^\circ}, \omega^n|_{X^\circ})$. The metric measure space $(\hat X, d_\omega, \omega^n)$ is extensively studied in \cite{GPSS2} both analytically and geometrically via complex Monge-Amp\`ere equations and its coupled Laplace equation.  The spectral theory is established in \cite{GPSS2} for $W^{1,2}(\hat X)$ along with many other uniform estimates for diameters, Sobolev inequalities, Green's functions and heat kernels. The series of papers \cite{FGS, S2, GPT, GPSS1, GPSS2, GPSS3} aim to build a framework that would expand the classical works \cite{Y1, Ko1} on complex Monge-Amp\`ere equations to geometric analysis on complex spaces with singularities. In particular, one wishes to establish connections between analysis on singular K\"ahler spaces and their topological, geometric and algebraic structures, which might require additional geometric or analytic assumptions beyond the Nash entropy. For example, one does not expect the Sobolev inequality with optimal exponents to hold for K\"ahler metrics in $\mathcal{V} (X, \theta_X, n, A,  p, K)$.

 A natural geometric assumption for singular K\"ahler metrics will be suitable bounds on the Ricci curvature. The Ricci curvature for $\omega \in \mathcal{V} (X, \theta_X, n, A,  p, K)$ can be defined as a current if the volume measure $\omega^n$ satisfies suitable quasi-plurisubharmonic conditions (c.f. Definition \ref{riclbdef}). When the Ricci curvature $\ric(\omega)$ is bounded below as a current, one would like to understand the geometric and analytic structures of $(\hat X, d_\omega, \omega^n)$. In particular, one would expect the notion of Ricci curvature bounded below in terms of the pluripotential theory should be equivalent to various synthetic Ricci curvature lower bounds in the study of RCD spaces developed by \cite{LV, Stk, AGS} and many others. The following specific questions arise naturally if the Ricci curvature of the singular K\"ahler space $(X, \omega)$ is bounded below.

  \begin{enumerate}
 
 \item Is $(\hat X, d_\omega)$ defined as in (\ref{mmsp}) homeomorphic to the original singular space $X$? 
 
 \smallskip
 
 \item Does the Sobolev inequality with optimal exponent hold on $(\hat X, d_\omega, \omega^n)$?

 \end{enumerate}
 
 It is obvious that these questions are related to the theory of RCD spaces due to recent progress in the compactness of K\"ahler manifolds with suitable Ricci curvature bounds, particularly in the case when $(X, \omega)$ admits suitable smoothing \cite{DS1, LS} or special resolution of singularities \cite{S2}. Recently , Question (1) and (2)  are confirmed in the following cases for a projective variety $X$ with log terminal singularities. 

\begin{enumerate}

\item   $(X, \omega)$ is K\"ahler-Einstein and $X$ admits a resolution $\pi: Y \rightarrow X$ with $\pi$-nef anti-canonical divisor $-K_Y$  \cite{Sz24},  

 \smallskip

\item $\dim_{\mathbb{C}} X =3$ \cite{FGS2}, 

 \smallskip

\item $(X, \omega)$ has Ricci curvature bounded below and $X$ admits a resolution $\pi: Y \rightarrow X$ with $\pi$-effective anti-canonical divisor $-K_Y$ \cite{FGS2}. 
 
\end{enumerate}

We will first establish new Sobolev estimates for   $\omega \in \mathcal{V} (X, \theta_X, n, A,  p, K)$ on a smooth K\"ahler manifold $X$ with bounded Calabi energy. These estimates will help us attack Question (1) and (2) and generalize the recent works of \cite{Sz24, FGS2}.  
 Recall that the Calabi energy for a K\"ahler metric $\omega$ on $X$ is defined by
 \begin{equation}\label{cala0}
 \mathcal{C}a(\omega)= \frac{1}{V_\omega} \int_X \left( \scR (\omega) \right) ^2 \omega^n, 
 \end{equation}
where $\scR(\omega)$ is the scalar curvature of $\omega$.  The following is our first main result.

\begin{theorem} \label{thm:main1} Let $(X, \theta_X)$ be an $n$-dimensional compact K\"ahler manifold equipped with a smooth K\"ahler metric $\theta_X$. Suppose $\omega$  is a K\"ahler metric in $\mathcal{V}(X, \theta_X, n, A,  p, K)$ with $p>n$ satisfying
$$\mathcal{C}a(\omega)   \leq B$$
and  $u\in C^\infty(X)$ is a solution of the Laplace equation
\begin{equation}\label{lap0}
\Delta_\omega u = f, ~ \int_X u\omega^n = 0
\end{equation}
for $f\in C^\infty(X)$.  Then the following hold. 

\begin{enumerate} 

\item There exists $C=C\left(X, \theta_X, n, A, B, p, K, \|f\|_{L^\infty(X)} \right)>0$ such that 
$$\|u\|_{L^\infty(X)} \leq C,$$
$$\frac{1}{V_\omega} \int_X \left( \abs{\nabla \nabla u}_\omega + \abs{\nabla \overline{\nabla} u}_\omega  +   |\nabla u|_\omega^4\right) \omega^n  \le C.$$

\medskip

\item For any $r>0$, there exists $C=C\left(X, \theta_X, n, A, B, p, K, \|f\|_{L^\infty(X)}, r \right)>0$ such that for any $x\in X$, 
$$\frac{1}{V_\omega} \int_{X\setminus B_\omega (x, r)} \left( V_\omega^2 \abs{\nabla \nabla G(x,\cdot)}_\omega + V_\omega^2\abs{\nabla \overline{\nabla} G(x,\cdot)}_\omega + V_\omega^4 |\nabla G(x,\cdot)|^4_\omega  \right) \omega^n \le C, $$
where  $B_\omega(x, r)$ is the $\omega$-geodesic ball centered at $x$ with radius $r$. 

\end{enumerate}

\end{theorem}

The equation (\ref{lap0}) was studied in \cite{GPSS2}, where uniform $W^{1,2}$ bounds are obtained for the solution $u$ to build the spectral theory on the metric measure space $(\hat X, d_\omega, \omega^n)$. The main contribution of Theorem \ref{thm:main1} is to establish uniform $W^{2,2}$ and $W^{1,4}$ estimates for $u$. Such analytic improvements are essential for geometric applications considered in this paper.

 One can consider the space $\nu(X, \theta_X, n , A, p, K; B) = \{ \omega\in \nu(X, \theta_X, n , A, p, K): ~ \mathcal{C}a(\omega) \leq B\} $. The compactness problem of $\nu(X, \theta_X, n , A, p, K; B)$ is rather complicated even in the non-collapsing case when the volume $V_\omega$ is uniformly bounded above away from $0$. Instead, we would like to investigate a singular K\"ahler space $(X, \omega)$ that can be approximated by smooth K\"ahler metrics in $\nu(X, \theta_X, n , A, p, K; B)$ and to establish geometric and topological structures for the metric measure space $(\hat X, d_\omega, \omega^n)$.

Suppose $X$ is an $n$-dimensional compact normal K\"ahler variety. We choose $\pi: Y \rightarrow X$, a resolution of singularities for $X$, where the K\"ahler manifold $Y$ is a nonsingular   model of $X$. Let $\theta_X$ be a fixed smooth background K\"ahler metric $\theta_X$  on $X$.  We therefore introduce the following geometric regularization for singular K\"ahler metrics on the  singular  space $X$. 
  
\begin{definition} \label{calareg} 
Let  $X$ be an $n$-dimensional compact normal K\"ahler variety equipped with a smooth K\"ahler metric $\theta_X$. A singular K\"ahler metric $\omega \in  \mathcal{V}(X, \theta_X, n, A, p, K)$ with $p>n$ is said to admit a regularization with uniform Nash entropy and Calabi energy if there exists a resolution $\pi: Y \rightarrow X$ with the following properties.

\begin{enumerate}

\item There exist a sequence of smooth K\"ahler metrics $\omega_j$  and a K\"ahler class $\alpha$ on $Y$ such that
\begin{equation}\label{cls1} 
[\omega_j] = \epsilon_j \alpha + [\pi^*\omega],~\epsilon_j \rightarrow 0^+,  ~ \omega_j \rightarrow \pi^*\omega~ {\rm in}~ C^2_{loc}(\pi^{-1}(X^\circ)). 
\end{equation}

%\smallskip

\item There exist $p'>n$ and $K'>0$ such that for all $j$, 
\begin{equation}\label{nashb1}
\mathcal{N}_{\theta_X, p'}(\omega_j) = \frac{1}{V_{\omega_j}}  \int_Y \left|  \log \left( V_{\omega_j}^{-1} \frac{\omega_j^n}{ (\pi^*\theta_X)^n} \right)\right|^{p'} \omega_j^n \leq K'. 
\end{equation}

\smallskip

\item There exists $B>0$ such that for all $j$, 
\begin{equation}\label{scb1}
\frac{1}{V_{\omega_j}}  \int_Y \left( \scR(\omega_j) \right)^2 \omega_j^n \leq B.
\end{equation}

\end{enumerate}

\end{definition}

Our next result identifies $(\hat X, d_\omega, \omega^n)$ as an RCD space if the singular K\"ahler metric admits a regularization as in Definition \ref{calareg} as well as a uniform lower bound for its Ricci curvature. 

\begin{theorem}\label{thm:main2} Let $X$ be an $n$-dimensional compact normal K\"ahler variety.   Suppose a singular K\"ahler metric $\omega  \in \mathcal{V}(X, \theta_X, n, A, p, K)$ with $p>n$   satisfies the following properties.

\begin{enumerate}

\item $(X, \omega)$ admits a regularization $(Y, \{\omega_j\}_{j=1}^\infty)$ satisfying \eqref{cls1}, \eqref{nashb1} and \eqref{scb1}, % (\ref{cls2}),  (\ref{nashb2}) and (\ref{ricp2}), 

\smallskip

\item the Ricci curvature of $\omega$ satisfies
$$\ric(\omega) \geq - \omega $$
on $X^\circ$. 

\end{enumerate}
Then the metric measure space $(\hat X, d_\omega, \omega^n)$ (as in \eqref{mmsp}) is a  non-collapsed compact ${\rm RCD}(-1,2n)$ space. Furthermore,  there exists $c>0$ such that 
\begin{equation}
\omega \geq c \theta_X.
\end{equation}

\end{theorem}

A similar result to Theorem \ref{thm:main2} has been established in \cite{CCHSTT} for K\"ahler-Einstein currents. One can further prove that $\hat X$ is homeomorphic to the original variety $X$ using techniques in \cite{Sz24} and \cite{FGS2} assuming that $[\omega] \in H^{1,1}(X, \mathbb{Q})$ and the Ricci curvature $\ric(\omega)$ is bounded. However, it is in general challenging to construct a regularization as in Definition \ref{calareg} for $\omega \in \nu(X, \theta_X, n, A, p, K)$. Fortunately,  it can be accomplished if the anti-canonical line bundle $-K_Y$ is $\pi$-nef for the resolution $\pi: Y \rightarrow X$. This special case is first considered in \cite{Sz24} to approximate singular K\"ahler-Einstein metrics by smooth cscK metrics on the nonsingular model. Our next result generalizes the work of \cite{Sz24} on singular K\"ahler-Einstein metrics to singular K\"ahler metrics with Ricci curvature bounded below.

\begin{theorem} \label{thm:main3} Let $X$ be an $n$-dimensional compact normal K\"ahler variety with log terminal singularities. Suppose %
\begin{enumerate}

\item there exists a resolution $\pi: Y \rightarrow X$ such that   $-K_Y$ is  $\pi$-nef or $\pi$-effective, 

\smallskip

\item $\omega\in H^{1,1}(X, \mathbb{R})\cap H^2(X, \mathbb{Q})$ is a singular K\"ahler metric in $\mathcal{V}(X, \theta_X, n, A, p, K)$ with $p>n$ satisfying 
$$\ric(\omega) \geq - \omega$$ 
on $X$ globally as currents. 
\end{enumerate}
Then  the following hold. 

\begin{enumerate}

\item The metric measure space $(\hat X, d_\omega, \omega^n)$ (as in \eqref{mmsp}) is a  non-collapsed compact ${\rm RCD}(-1,2n)$ space homeomorphic to $X$, 

\smallskip

\item The regular set $\cR(\hat X)$ of $\hat X$ coincides with $X^\circ$, and the singular set $\cS(\hat X)$  has Hausdorff dimension no great than $2n-3$.

\smallskip

\item There exists $c>0$ such that $\omega\geq c\theta_X$. In particular, the identity map $\iota: (\hat X, d_\omega) \rightarrow (X, \theta_X)$ is Lipschitz.

\end{enumerate}

\end{theorem}

We remark that Theorem \ref{thm:main3} is proved in \cite{FGS2} in the case when $-K_Y$ is $\pi$-effective. Theorem \ref{thm:main3} provide abundant examples of RCD spaces from algebraic geometry. For any projective variety $X$ as in Theorem \ref{thm:main3} and any ample line bundle $L \rightarrow X$, one can always construct infinitely many K\"ahler currents in $c_1(L)$ with Ricci curvature bounded below, which induce RCD spaces topologically and holomorphically equivalent to $X$.  There are many examples of these varieties such as those  that admit a crepant resolution \cite{S2} or those of Fano type \cite{Sz24, BJT}. If the Ricci curvature of $\omega$ in Theorem \ref{thm:main3} is bounded, the singular set of $\hat X$ will have Hausdorff dimension no greater than $2n-4$. 

Finally, we would like to point out that our main estimates in Theorem \ref{thm:main1} can also be applied to special families of collapsing K\"ahler metrics on Calabi-Yau fibrations, whose limits are singular  K\"ahler-Einstein metrics twisted with a Weil-Petersson metric introduced in \cite{ST1, ST2}.

\section{Uniform linear estimates}

Let $X$ be an $n$-dimensional compact K\"ahler manifold equipped with a fixed background K\"ahler metric $\theta_X$. We define $\mathcal{V}_B (X, \theta_X, n, A, p, K) \subset \mathcal{V} (X, \theta_X, n, A, p, K) $ with $p>n$ to be the set of  smooth K\"ahler metrics $\omega$ with 
\begin{equation}\label{l2ric}
\frac{1}{V_\omega}  \int_X |\ric(\omega)_{-} |^2 \omega^n \leq B, 
\end{equation}
where $\ric(\omega)_-$ denotes the absolute value of the smallest negative eigenvalue of $\ric(\omega)$ and it is defined to be zero if $\ric(\omega)\ge 0$.  We observe that 
\begin{eqnarray}\label{cal2ric}
\frac{1}{ V_\omega} \int_X |\ric(\omega)_{-} |^2 \omega^n &=& \frac{1}{ V_\omega} \int_X |\scR(\omega)|^2 \omega^n - \frac{1}{ V_\omega} \int_X \ric(\omega)\wedge \ric(\omega)\wedge \omega^{n-2} \\
&=& \mathcal{C}a(\omega) - [K_X]^2\cdot [\omega]^{n-2}.
\end{eqnarray}
Hence the assumption (\ref{l2ric}) is equivalent to a uniform bound for the Calabi energy for $\omega \in \mathcal{V} (X, \theta_X, n, A, p, K) $.

We consider the Laplace equation $\Delta_\omega u = f$
on $(X, \omega)$ for $f\in C^\infty(X)$, where $\Delta_\omega$ is the Laplace operator associated to the K\"ahler metric $\omega$. It is obvious that $\int_X f \omega^n = 0$.  We will begin our proof for Theorem \ref{thm:main1} with the following lemma. 

\begin{lemma} \label{lemma L infty u}

 Let $\omega \in \mathcal{V}(X, \theta_X, n, A, B, p, K)$.  Suppose 
 $$
\Delta u = f, ~\int_X u \omega^n=0
$$
for $u, f\in C^\infty(X)$.  Then there exists $C=C(X, \theta_X, n, A, B, p, K, \|f\|_{L^\infty(X)})>0$ such that 
\begin{equation}
\|u\|_{L^\infty(X)} +\frac{1}{ V_\omega} \int_X | \nabla u|^2_\omega \omega^n \leq C.
\end{equation}

\end{lemma}
\begin{proof}
This follows immediately from the uniform $L^q(X,\omega)$-estimates \cite{GPSS1} on Green's function $G$ of $(X,\omega)$, for some $q=q(X, \theta_X, n, A, p, K)>1$. Indeed, for any point $x\in X$
$$u(x) = \frac{1}{V_\omega} \int_X u \omega^n + \int_X (-G (x,\cdot)) \Delta_\omega u \omega^n = \int_X (-G (x,\cdot)) f \omega^n.$$ The $L^\infty$-bound of $u$ follows from this by taking supremum over all $x\in X$. Multiplying the equation $\Delta u  = f$ by $u$ and integrating it over $(X,\omega^n)$ gives the $L^2(X,\omega)$-bound of $\nabla u$. 
\end{proof}

If fact, if the volume measure ratio is in $L^p$ for some $p>1$, then the $L^\infty$-bound for $f$ in the assumption of Lemma \ref{lemma L infty u} can be replaced by the $L^\gamma(X,\omega)$-bound for some $\gamma>n$. 
The following proposition combined with Lemma \ref{lemma L infty u} will imply the first part of Theorem \ref{thm:main1}.

\begin{prop} \label{linest1} Let $\omega \in \mathcal{V}_B(X, \theta_X, n, A, p, K)$.  Suppose \begin{equation}\label{lapeqn2}
\Delta_\omega u = f 
\end{equation}
for $u, f\in C^\infty(X)$. Then there exists $C=C(X, \theta_X, n, A, B, p, K, \|f\|_{L^\infty(X)})>0$ such that 
\begin{equation}\label{linear1}
\frac{1}{V_\omega} \int_X \left( \abs{\nabla \nabla u}_\omega + \abs{\nabla \overline{\nabla} u}_\omega  +   |\nabla u|_\omega^4\right) \omega^n  \le C, 
\end{equation}
where $\nabla$ and $\overline{\nabla}$  are covariant derivatives associated to $\omega$. 

\end{prop}

\begin{proof}  First, we can assume that $u\geq 1$ by replacing $u$ by $u- \inf_X u +1$. $\|u\|_{L^\infty(X)}$ is still uniformly bounded.  For convenience, we drop the underscore $\omega$  for norms of the covariant derivatives in the proof.  We apply the Bochner's formula to $|\nabla u|^2$ and have 
\begin{equation}\label{bocheqn}
\Delta \abs{\nabla u} = \abs{\nabla \nabla u } + \abs{\nabla \overline{\nabla} u } -  2 \innpro{\nabla u, \nabla f}_{\omega} + 2 \ric_{\omega} (\nabla u, \overline{\nabla} u).
\end{equation}
We let $\phi := u ^\beta$ for some $\beta\in (0,1]$ to be determined.  Then
\begin{equation}\label{eqn:s2-10}
\Delta \phi = - \beta u^{\beta - 1} f + \beta( \beta - 1) u^{\beta - 2} \abs{\nabla u} = - \beta \phi^{1-\frac{1}{\beta}} f + \frac{\beta - 1}{\beta} \phi^{-1} \abs{\nabla \phi}.
\end{equation} 
Applying the Bochner's formula again, we have
\begin{align*}
\Delta \abs{\nabla \phi} \ge &~  \abs{\nabla \nabla \phi} + \abs{\nabla \overline{\nabla} \phi} + 2\innpro{\nabla \phi , \nabla \big(- \beta \phi^{1-\frac{1}{\beta}} f + \frac{\beta - 1}{\beta} \phi ^{-1} \abs{\nabla \phi}\big)} - 2\ric(\omega)_- \abs{\nabla \phi}\\
\ge &~ \abs{\nabla \nabla \phi }  + \abs{\nabla \overline{\nabla} \phi }  - 2\beta \innpro{\nabla \phi , \nabla ( \phi ^{1-\frac{1}{\beta}} f)}  -2 \ric(\omega )_- \abs{\nabla \phi } \\
&~ +\frac{2(1-\beta)}{\beta} \phi ^{-2} |\nabla \phi | ^4  - \frac{4(1-\beta)}{\beta} \phi ^{-1} \abs{\nabla \phi }  \cdot |\nabla |\nabla \phi \|  .
\end{align*}
For the last term, we have
\begin{align*}
\frac{4(1-\beta)}{\beta} \phi ^{-1} \abs{\nabla \phi }  \cdot |\nabla |\nabla \phi \|  \le &~  \frac{1}{2} |\nabla |\nabla \phi \| ^2 + \frac{8(1-\beta)^2}{\beta^2} \phi ^{-2} |{\nabla \phi }|^4 \\
\le &~ \frac{1}{4} ( \abs{\nabla \nabla \phi }  + \abs{\nabla \overline{\nabla} \phi } ) + \frac{1-\beta}{\beta} \phi ^{-2} |{\nabla \phi }|^4 ,
\end{align*}
where in the second inequality we use the Kato's inequality and choose $\beta$ close enough to $1$, and in fact, $\beta = 8/9$ suffices. The term involving the Ricci curvature  satisfies
\begin{equation*}
2 \ric(\omega )_- \abs{\nabla \phi } \le \frac{1-\beta}{2\beta} \phi ^{-2} |\nabla \phi | ^4 + \frac{2\beta}{1-\beta} \phi ^2 |\ric(\omega )_-|^2.
\end{equation*}
Combining these, we get
\begin{align*}
\Delta \abs{\nabla \phi }    
\ge &~ \frac{1}{2}(\abs{\nabla \nabla \phi }  + \abs{\nabla \overline{\nabla} \phi }  ) - 2\beta \innpro{\nabla \phi , \nabla ( \phi ^{1-\frac{1}{\beta}} f)}  
+\frac{(1-\beta)}{2 \beta} \phi ^{-2} |\nabla \phi | ^4 \\
&~ - \frac{2\beta}{1-\beta} \phi ^2 |\ric(\omega )_-|^2 
\end{align*}
Integrating this inequality over $(X, \omega ^n)$ yields
\begin{align*}
&~ \frac{1}{V_\omega} \int_X \bk{ \frac{1}{2}(\abs{\nabla \nabla \phi }  + \abs{\nabla \overline{\nabla} \phi }  ) 
+  \frac{(1-\beta)}{2 \beta} \phi ^{-2} |\nabla \phi | ^4} \omega^n \\
\le&~\frac{1}{V_\omega}   \int_X \big ( 2\beta \innpro{\nabla \phi , \nabla ( \phi ^{1-\frac{1}{\beta}} f)}  + \frac{2\beta}{1-\beta} \phi ^2 |\ric(\omega )_-|^2  \big)\omega ^n \\
\le &~ - \frac{1}{V_\omega}  \int_X 2\beta \phi ^{1-\frac{1}{\beta}} f  \Delta  \phi  + CB   \\
=&~-  \frac{1}{V_\omega}  \int_X 2\beta \phi ^{1-\frac{1}{\beta}} f ( - \beta \phi ^{1-\frac{1}{\beta}} f + \frac{\beta - 1}{\beta} \phi ^{-1} \abs{\nabla \phi } ) + C B\\
\le &~ C,
\end{align*}
where in the last step we use the fact that $\phi $ and $f$ are both uniformly bounded. Since $\nabla \phi  = \beta u^{\beta - 1} \nabla u$, and $ \nabla^2 \phi  = \beta u^{\beta - 1} \nabla^2 u + \beta (\beta - 1) u^{\beta - 2} \nabla u \otimes \nabla u  $, we see that %This yields the estimate
\begin{equation*}%\label{eqn:n1.14}
\frac{1}{V_\omega}  \int_X \left(\abs{\nabla \nabla u}  + \abs{\nabla \overline{\nabla} u }  +   |\nabla u | ^4\right)\omega ^n  \le C.
\end{equation*}
This completes the proof of Proposition \ref{linest1}.
\end{proof}
\begin{remark}\label{remark 2.1}
By a very similar argument as in Proposition \ref{linest1}, one can obtain a slight generalization for the linear estimate for $u$ by only assuming a $L^{2-\epsilon}$-bound for $\ric(\omega)_-$  for a small $\epsilon>0$. In this case, one would have  
\begin{equation}\label{eqn:weak1}\int_X \left( |\nabla u|^{2 (2- \delta)}_\omega + \frac{ \abs{\nabla \overline{\nabla} u  }_\omega + \abs{\nabla \nabla u} _\omega  }{|\nabla u|^{2\delta}_\omega}    \right)\omega^n \le C. 
\end{equation}
for some small $\delta = \delta(\epsilon)>0$.

\end{remark}

%\section{Estimates for Green's function}

We now complete the proof of Theorem \ref{thm:main1} by deriving uniform Sobolev estimates for Green's functions.

\begin{prop}\label{prop 2-3} Let $\omega \in \mathcal{V}(X, \theta_X, n, A, B, p, K)$. For any $x \in X$  and $r>0$,  
there exists a $C=C(X, \theta_X, n, A, B, p, K,  r)>0$ 
 such that 
\begin{equation} \label{greenes1}
\frac{1}{V_\omega} \int_{X\setminus B_{\omega}(x, r)} \left(V_\omega^2  \abs{\nabla \nabla G(x,\cdot)}_\omega + V_\omega^2 \abs{\nabla \overline{\nabla} G(x,\cdot)}_\omega + V_\omega^4|\nabla G(x,\cdot)|^4_\omega  \right) \omega^n \le C.
\end{equation}
\end{prop}
\begin{proof} Let $U_s = X\setminus B_\omega(x, s)$ for $s>0$. By \cite{GPSS1}, we may also assume for a fixed $b>>1$, 
\begin{equation}\label{eqn:n1.7}
\frac{1}{V_\omega}  \le G(x,y) \le \frac{1}{V_\omega}  \frac{C_b}{d_\omega (x,y)^{ b}},\quad \forall x\neq y \in X.
\end{equation}

Consider the function $h (y): = V_\omega G (x,y)$, which satisfies the equation 
$$\Delta  h (y) = 1  $$
 for $y \in  U_{r}$. We have by \eqref{eqn:n1.7}
\begin{equation}\label{eqn:new 1.15}
1\le h  \le   C, \quad \text{in } U_{r},
\end{equation}
where all the constants $C$  depend on $r$. We choose a cut-off function $\rho_1$ satisfying
$$ \rho_1|_{U_{2r}}=1, ~ \rho_1|_{X\setminus U_r} =0, ~ |\nabla \rho_1|_{\omega } \le \frac{4}{ r}.$$ 
Multiplying $h  \rho_1^2$ on both sides of the equation $\Delta  h  = 1/V_\omega $ and applying integration by parts, we obtain 
\begin{align*}
-\frac{1}{V_\omega} \int_{U_{r}} h  \rho_1^2 \omega ^n =&~ \frac{1}{V_\omega} \int_{ U_{r}}\rho_1^2 |\nabla h |_{\omega }^2  + 2 \rho_1 h  \innpro{\nabla \rho_1, \nabla h } \\
\ge &~ \frac{1}{V_\omega} \int_{ U_{r}} \frac{1}{2}\rho_1^2 |\nabla h |_{\omega }^2  - 4  h ^2 |\nabla \rho_1| ^2. % \innpro{\nabla \rho_1, \nabla h } 
\end{align*}
Hence, we have
\begin{equation}\label{eqn:new1.16}
\frac{1}{V_\omega}\int_{ U_{2r}} |\nabla h |_{\omega }^2 \omega ^n \le \frac{1}{V_\omega} \int_{ U_{r}}\rho_1^2 |\nabla h |_{\omega }^2 \omega ^n \le C_x.
\end{equation}

We write $\psi  := h ^\beta$ for $\beta = 8/9$. We have
\begin{equation}\label{eqn:n1.11}
\Delta  \psi  =  \beta  h ^{\beta - 1}  + \beta( \beta - 1) h ^{\beta - 2} \abs{\nabla h  }  = \beta  \psi ^{1 - \frac{1}{\beta}} + \frac{\beta - 1}{\beta} \psi ^{-1} \abs{\nabla \psi } .
\end{equation} By the Bochner's formula, we obtain that
\begin{align*}
\Delta  \abs{\nabla \psi }  \ge &~  \abs{\nabla \nabla \psi }  + \abs{\nabla \overline{\nabla} \psi }  + 2\innpro{\nabla \psi , \nabla \big(\beta  \psi ^{1- \frac{1}{\beta} } + \frac{\beta - 1}{\beta} \psi ^{-1} \abs{\nabla \psi } \big)}  - 2\ric(\omega )_- \abs{\nabla \psi } \\
\ge &~ \abs{\nabla \nabla \psi }  + \abs{\nabla \overline{\nabla} \psi }  + 2\beta  \innpro{\nabla \psi , \nabla ( \psi ^{1-\frac{1}{\beta}})}  -2 \ric(\omega )_- \abs{\nabla \psi } \\
&~ +\frac{2(1-\beta)}{\beta} \psi ^{-2} |\nabla \psi | ^4  - \frac{4(1-\beta)}{\beta} \psi ^{-1} \abs{\nabla \psi }  \cdot |\nabla |\nabla \psi \|  .%\innpro{\nabla \psi ,  \frac{\beta - 1}{\beta} \psi ^{-1} \abs{\nabla \psi } }  
\end{align*}
We estimate the  last term as follows
\begin{align*}
\frac{4(1-\beta)}{\beta} \psi ^{-1} \abs{\nabla \psi }  \cdot |\nabla |\nabla \psi \|  \le &~  \frac{1}{2} |\nabla |\nabla \psi \| ^2 + \frac{8(1-\beta)^2}{\beta^2} \psi ^{-2} |{\nabla \psi }|^4 \\
\le &~ \frac{1}{4} ( \abs{\nabla \nabla \psi }  + \abs{\nabla \overline{\nabla} \psi } ) + \frac{1-\beta}{\beta} \psi ^{-2} |{\nabla \psi }|^4 ,
\end{align*}
where in the second inequality we use the Kato's inequality. The term involving the Ricci curvature  satisfies
\begin{equation*}
2 \ric(\omega )_- \abs{\nabla \psi } \le \frac{1-\beta}{2\beta} \psi ^{-2} |\nabla \psi | ^4 + \frac{2\beta}{1-\beta} \psi ^2 |\ric(\omega )_-|^2.
\end{equation*}
Combining these, we get
\begin{equation}\label{eqn:new1.18}\begin{split}
\Delta  \abs{\nabla \psi }    
\ge &~ \frac{1}{2}(\abs{\nabla \nabla \psi }  + \abs{\nabla \overline{\nabla} \psi }  ) +  2\beta  \innpro{\nabla \psi , \nabla  \psi ^{1-\frac{1}{\beta}}}  
+\frac{(1-\beta)}{2 \beta} \psi ^{-2} |\nabla \psi | ^4 \\
&~ - \frac{2\beta}{1-\beta} \psi ^2 |\ric(\omega )_-|^2. 
\end{split}\end{equation}
We now choose a cut-off function $\rho_2$ satisfying
$$ \rho_2|_{U_{3r}}=1, ~ \rho_2|_{X\setminus U_{2r}} =0, ~ |\nabla \rho_2|_{\omega } \le \frac{4}{ r}.$$ 
We multiply both sides of \eqref{eqn:new1.18} by $\rho_2^2$ and integrate the resulted inequality over $ U_{2r}$. We obtain 
\begin{equation}\label{eqn:new1.19}\begin{split}
&~  \frac{1}{V_\omega} \int_{ U_{2r}} \rho_2^2 \Big(\frac{1}{2}(\abs{\nabla \nabla \psi }  + \abs{\nabla \overline{\nabla} \psi }  ) 
+\frac{(1-\beta)}{2 \beta} \psi ^{-2} |\nabla \psi | ^4\Big) \omega ^n \\
\le &~  \frac{1}{V_\omega} \int_{ U_{2r}}  \frac{2\beta}{1-\beta} \rho_2^2 \psi ^2 |\ric(\omega )_-|^2 - 
2\beta \rho_2^2 \innpro{\nabla \psi , \nabla  \psi ^{1-\frac{1}{\beta}}}   + \rho_2^2 \Delta  \abs{\nabla \psi } .
\end{split}\end{equation}
The first term on the RHS of \eqref{eqn:new1.19} is bounded by $C B$. The second term satisfies 
$$ \int_{ U_{2r}}    
\frac{2(1-\beta)}{V_\omega}\rho_2^2 \psi ^{-\frac{1}{\beta}}\abs{\nabla \psi }  \omega^n \le C,$$
by \eqref{eqn:new1.16} and the equation $\nabla \psi  = \beta h ^{\beta - 1} \nabla h $. The last term satisfies 
\begin{align*}
 \frac{1}{V_\omega} \int_{ U_{2r}}   \rho_2^2 \Delta  \abs{\nabla \psi } \omega^n &~ = -   \frac{1}{V_\omega} \int_{U_{2r}}   4 \rho_2 |\nabla \psi |  \cdot \innpro{\nabla \rho_2,  \nabla|{\nabla \psi }| }  \omega^n \\
&~\le   \frac{1}{V_\omega} \int_{ U_{2r}}   4 \rho_2 |\nabla \psi |  \cdot |\nabla \rho_2|  \cdot |\nabla|{\nabla \psi| }|  \omega^n\\
&~\le   \frac{1}{V_\omega} \int_{ U_{2r}}  \left( \frac{1}{2}\rho_2^2|\nabla|{\nabla \psi |}|  ^2 + 
  8  |\nabla \psi |  ^2\cdot |\nabla \rho_2|  ^2  \right)\omega^n\\
  &~\le    \frac{1}{V_\omega}\int_{ U_{2r}}  \frac{1}{4}\rho_2^2(\abs{\nabla \nabla\psi }  + \abs{\nabla\overline{\nabla} \psi } ) \omega^n+ C.
   \end{align*}
Combining the above estimates, we obtain
\begin{align*}%{equation*}\begin{split}
&~ \frac{1}{V_\omega} \int_{ U_{3r}} \Big(\frac{1}{4}(\abs{\nabla \nabla \psi }  + \abs{\nabla \overline{\nabla} \psi }  ) 
+\frac{(1-\beta)}{2 \beta} \psi ^{-2} |\nabla \psi | ^4\Big) \omega ^n \\
\le &~ \frac{1}{V_\omega}\int_{ U_{3 r}} \rho_2^2 \Big(\frac{1}{4}(\abs{\nabla \nabla \psi }  + \abs{\nabla \overline{\nabla} \psi }  ) 
+\frac{(1-\beta)}{2 \beta} \psi ^{-2} |\nabla \psi | ^4\Big) \omega ^n\\
\le &~ C.
\end{align*}
This directly implies   that 
$$ \frac{1}{V_\omega} \int_{U_{3r}} \Big(\abs{\nabla \nabla h }  + \abs{\nabla \overline{\nabla} h }   
+ |\nabla h | ^4\Big) \omega ^n\le C$$ 
and we have completed the proof of Proposition by replacing $r$ by $3r$. 
\end{proof}

\section{Regularization and the Lipchitz property for eigenfunctions}

We will now apply the Sobolev estimates of Theorem \ref{thm:main1} to establish analytic estimates on singular  K\"ahler spaces. In this section, we will always assume that $X$ is an $n$-dimensional compact normal K\"ahler variety. 

\begin{definition}\label{smm} A positive closed $(1,1)$-current $\theta$ on $X$ is said to be a smooth K\"ahler metric if for any $p\in X$, there exists a local holomorphic embedding $\iota: U \rightarrow \mathbb{C}^N$ from an open neighborhood $U$ of $p$ such that $\theta$ is the restriction of a smooth K\"ahler metric $\tilde \theta$  in an open neighborhood $\iota(U) \subset \mathbb{C}^N$. 

\end{definition}

There always exist smooth K\"ahler metrics on projective varieties via projective embbedings. Let $\pi: Y \rightarrow X$ be a resolution singularities for $X$. We fix smooth background K\"ahler metrics $\theta_X$ and $\theta_Y$ on $X$ and $Y$ respectively.  Let $\omega \in C^\infty(X^\circ) \cap \nu(X, \theta_X, n, A, p, K)$ be a singular K\"ahler metric on $X$ with $p>n$.  we assume that $\omega$ admits the following regularization. 

\begin{enumerate}

\item There exist a sequence smooth K\"ahler metrics $\omega_j$ and a K\"ahler class $\alpha$ on $Y$ satisfying
\begin{equation}\label{cls2} 
[\omega_j] = [\pi^*\omega]+\epsilon_j \alpha, ~\epsilon_j \rightarrow 0^+, ~ \omega_j \rightarrow \pi^*\omega~ {\rm in}~ C^2_{loc}(\pi^{-1}(X^\circ)). 
\end{equation}

\item There exists $p'>n$ and $K'>0$ such that for all $j$, 
\begin{equation}\label{nashb2}
\mathcal{N}_Y(\omega_j) =\frac{1}{V_{\omega_j}} \int_Y \left| \log \left( V_{\omega_j}^{-1}\frac{\omega_j^n}{\theta_Y^n} \right) \right|^{p'} \omega_j^n \leq K' .
\end{equation}

\item There exists $B>0$ such that for all $j$, 
\begin{equation}\label{ricp2}
\int_Y |\ric(\omega_j)_{-}|^2 \omega_j^n \leq B.
\end{equation}

\end{enumerate}

The following is the main result in this section. 

\begin{prop} \label{thm:lip2} Let $X$ be normal K\"ahler variety with log terminal singularities of $\dim_{\mathbb{C}}(X)=n$.   Suppose $\omega \in \mathcal{V}(X, \theta_X, n, A, p, K) $ with $p>n$ is a singular K\"ahler metric satisfying the following properties.

\smallskip

\begin{enumerate}

\item $\ric(\omega) \geq - \omega$ in $X^\circ$.

\medskip

\item $\omega$ admits a regularization $(Y, \{\omega_j\}_{j=1}^\infty)$ satisfying (\ref{cls2}),  (\ref{nashb2}) and (\ref{ricp2}). 

\end{enumerate}
\medskip
Then every eigenfunction of $\Delta_\omega$ is Lipschitz.

\end{prop}

We begin our proof for Proposition \ref{thm:lip2} with the following lemma.

\begin{lemma} \label{catop}  Suppose $(Y, \{\omega_j\}_{j=1}^\infty)$ satisfy  (\ref{cls2}),  (\ref{nashb2}). If there exists $B>0$ such that 
\begin{equation}\label{rb2}
\mathcal{C}_a(\omega_j) =\frac{1}{V_{\omega_j}} \int_Y (\scR(\omega_j))^2 \omega_j^n \leq B,  
\end{equation}
then there exists $B'>0$ such that 
$$\int_Y |\ric(\omega_j)|^2 \omega_j^n \leq B.$$

\end{lemma}

\begin{proof}  By (\ref{cal2ric}), we have 
\begin{eqnarray*}
\int_Y |\ric(\omega_j)|^2 \omega_j^n &~=&~ \int (\scR(\omega_j))^2 \omega^n - \int_X \ric(\omega_j)\wedge \ric(\omega_j)\wedge \omega_j^{n-2} \\
&~=&~ \int (\scR(\omega_j))^2 \omega^n - [K_Y]^2\cdot [\omega_j]^{n-2} \leq \int (\scR(\omega_j))^2 \omega^n + C 
\end{eqnarray*}
for some uniform $C>0$. \end{proof}

  Suppose $u \in W^{1,2}(X, \omega)$ is an eigenfunction of $\Delta_\omega$ satisfying 
$$\Delta_\omega u = - \lambda u $$
for some  $\lambda >0$. The following estimates are proved in \cite{GPSS2} (c.f. Lemma 11.2 in \cite{GPSS2}).

\begin{lemma} $u\in L^\infty(X) \cap C^\infty(X^\circ)$.

\end{lemma}

Since the singular set $\cS(X)$ is an analytic subvariety of $X$,   the following result of J. Sturm  \cite{St}  holds due to the uniform $L^\infty$-bound for the K\"ahler potentials of $\omega$ and $\omega_j$. 

 \begin{lemma} \label{cutoff} $\cS(X)$ has capacity $0$ with respect to $\omega$ and $\omega_j$ for all $j$. More precisely, there exist a sequence of smooth cut-off functions $\{\rho_k\}_{k=1}^\infty$ satisfying the following conditions.
 
 \begin{enumerate}
 
 \item There exist a sequence of relatively compact domains $\{ \Omega_k \}_{k=1}^\infty$ of $X^\circ$ with 
 $$\Omega_1\subset\subset \Omega_2 \subset\subset ... \subset\subset \Omega_k ... \rightarrow X^\circ$$   and
 $$\rho_k|_{\Omega_k} = 1, ~ {\rm supp} \rho_k \subset\subset X^\circ,~ 0\leq \rho \leq 1. $$
 
 \item $$\lim_{k \rightarrow \infty} \int_X |\nabla \rho_k|_{\omega} ^2 \omega^n = \lim_{k \rightarrow \infty} \sup_{j\geq 1} \int_X |\nabla \rho_k|_{\omega_j} ^2 \omega_j^n = 0. $$

 \end{enumerate}

 \end{lemma}

Let $G=G(\cdot, \cdot)$ be Green's function of $(X,\omega)$ constructed as in \cite{GPSS2}. It satisfies the following estimates. 
\begin{lemma}
\label{lemma green limit}
There exist  $q = q(n,p)>1$ and $C= C(X,\theta_X, n,A, p, K)>0$ such that 
\begin{equation}\label{eqn:gr1}
\sup_{x\in X^\circ} \int_X  |G(x,\cdot)|^q \omega^n \le C,\quad \inf_{(x,y)\in X^\circ \times X^\circ } G(x,y)\ge -C.
\end{equation}
Moreover, Green's formula holds for any smooth function with compact support in $X^\circ$, i.e., for any $u\in C^\infty(X^\circ)$ with ${\rm supp} u \subset\subset X^\circ$, we have for any $x\in X^\circ$
\begin{equation}\label{eqn:gr2}
u(x) = \frac{1}{V_\omega} \int_X u \omega^n - \int_X G(x,\cdot) \Delta_\omega u \omega^n.
\end{equation}
\end{lemma} 
\begin{proof} We fix a point $x\in X^\circ$ and view Green's function $G(x,y)$ as a function of $y$. 
It follows from the main theorem of \cite{GPSS2} that for some uniform constants $q, C>0$, these estimates in \eqref{eqn:gr1} hold uniformly for Green's functions $G_{j}(x,\cdot)$ on $(Y, \omega_j)$. Since $\omega_j$ converges in $C^2_{loc}(X^\circ)$ to $\omega$, by elliptic regularity we see that $G_{j}(x,\cdot)$ converges uniformly in $C^3$ on any compact subset $K\subset\subset X^\circ\backslash \{x\}$ to Green's function $G(x,\cdot)$. Then the estimates in  \eqref{eqn:gr1} follow from Fatou's lemma. 

To verify \eqref{eqn:gr2}, we identify $u$ with its pullback $\pi^*u$ on $Y$, and $x$ with $\pi^{-1}(x)$. We apply Green's formula
$$u(x) = \frac{1}{V_{\omega_j}} \int_Y u \omega_j^n - \int_Y G_{j} (x,\cdot) \Delta_{\omega_j} u \omega_j^n, $$
where $V_{\omega_j} = [\omega_j]^n$. Passing to the limit, the first integral on the RHS converges to the first integral in \eqref{eqn:gr2}. For the second integral, we calculate (denote $\Omega\subset X^\circ$ to be the compact support of $u$)
\begin{align*}
& ~\Big| \int_Y G_{j} (x,\cdot) \Delta_{\omega_j} u\, \omega_j^n - \int_X G(x,\cdot) \Delta_\omega u\,\omega^n    \Big|\\
 \le & ~  \Big| \int_\Omega G_{j} (x,\cdot) \xk{ \Delta_{\omega_j} u\, \omega_j^n - \Delta_\omega u\,\omega^n   } \Big| + \Big| \int_\Omega \xk{ G_{j} (x,\cdot) -  G(x,\cdot)} \Delta_\omega u\,\omega^n    \Big|.
\end{align*}
The first term tends to $0$ as $j\to\infty$ since $\Delta_{\omega_j} u\, \omega_j^n \to \Delta_\omega u\, \omega^n$ on $\Omega$ while the $L^1$-norm of $G_{j}(x,\cdot)$ is uniformly bounded. For the second term, we choose a small neighborhood $U$ of $x$ whose $\omega^n$-volume is very small. Outside of $U$, $G_{j}(x,\cdot)$ converges uniformly to $G(x,\cdot)$. Hence this term can be made arbitrarily small by taking $j$ large enough and invoking the H\"older's inequality $\| G_{j} \|_{L^1(U,\omega^n)} \le \| G_{j} \|_{L^q(U,\omega^n)} \xk{ \mathrm{Vol}(U,\omega_j^n)  }^{1/q}$ and the uniform $L^q$ bound of $G_{j}(x,\cdot)$. From this, Green's formula \eqref{eqn:gr2} follows easily. 
\end{proof}
Due to the lower bound of $G(x,\cdot) $ in \eqref{eqn:gr1}, we may add a constant to $G(x,\cdot)$ and  assume 
$$G(x, \cdot) \geq 1. $$
  We can further establish the following uniform estimate for a weighted $L^2$ bound on the gradient of $G(x,\cdot)$: for any $\delta>0$, there exists a uniform constant $C_\delta>0$ such that
\begin{equation}\label{eqn:gr3}
\sup_{x\in X^\circ} \int_X \frac{ |\nabla G(x,\cdot)|^2   }{G(x,\cdot)^{1+\delta}} \omega^n \le C_\delta. 
\end{equation} 
The proof of \eqref{eqn:gr3} is similar to that of \eqref{eqn:gr1}, by applying Lemma 5.6 in \cite{GPSS2}, so we omit it. 

We will now apply our estimates to eigenfunctions. 

\begin{lemma}\label{lemma green} Suppose $u\in W^{1,2}(X, \omega)$ is an eigenfunction of $\Delta_\omega$ with eigenvalue $\lambda>0$.  Then   for any $x\in X^\circ$, we have
\begin{equation}\label{eqn:n1.4}
u(x) = -\int_X G(x, \cdot) \Delta u \omega^n = \lambda \int_X G(x,y) u(y) \omega^n(y) . 
\end{equation}
\end{lemma}
\begin{proof} We first apply the standard Green's formula \eqref{eqn:gr2} to the function $u\rho_k$, where $\rho_k$ is the cut-off function constructed in Lemma \ref{cutoff}. Obviously, $u \rho_k $   has compact support in $X^\circ$, and 
\begin{align}\label{eqn:n1.5}
u\rho_k(x) &= -\int_X G(x, \cdot)( \rho_k \Delta u + u \Delta\rho_k + 2 \langle \nabla u , \nabla \rho_k\rangle ) \omega^n + \frac{1}{V_\omega} \int_X u\rho_k \omega^n \\
& =: I_1 + I_2, \nonumber
\end{align}
where $I_1$ and $I_2$ denote the first and second integrals, respectively. 
As $k\to\infty$, the LHS and the first term in $I_1$   of \eqref{eqn:n1.5} converge to the corresponding terms in \eqref{eqn:n1.4}. $I_2$ clearly tends to the integral of $u$ which is zero. We need to check the remaining terms tend to zero. 

For the last term in $I_1$ of \eqref{eqn:n1.5},  $x$ is disjoint from $\mathrm{supp}(\nabla \rho_k)$ when $k$ is sufficiently large. Hence there exists a constant $C_x$ which depends on $x$, such that $|G(x, y)|\le C_x$ by \cite{GPSS1}, for any $y\in \mathrm{supp}(\nabla \rho_k)$. Then this term tends to zero by H\"older's inequality, the $L^2$-bound of $\nabla u$ and Lemma \ref{cutoff}. 

For the second term in $I_1$, we apply integration by parts to obtain 
$$ -\int_X G(x, \cdot)u \Delta\rho_k \omega^n = \int_X  \left(  G(x, \cdot) \innpro{\nabla u, \nabla \rho_k} + u \innpro{ \nabla G(x,\cdot) ,\nabla \rho_k  } \right) \omega^n.$$ 
Then by the same reasoning as in the previous term, we have $$\int_X  G(x, \cdot) \innpro{\nabla u, \nabla \rho_k} \omega^n \leq C \left( \int_X  (G(x, \cdot) )^2 |\nabla \rho_k|^2 \omega^n \right)^{1/2} \rightarrow 0 .$$ 
Moreover, 
$$\left| \int_X u \innpro{ \nabla G(x,\cdot) ,\nabla \rho_k  }\right| \le C \left(\int_X \frac{|\nabla G(x,\cdot)|^2}{G(x,\cdot)^{1+\delta}} \omega^n \right)^{1/2} \left(\int_X G(x,\cdot)^{1+\delta} |\nabla \rho_k|^2 \omega^n\right)^{1/2}\to 0$$
where the first factor on the RHS is bounded by \eqref{eqn:gr3}. This completes the proof of the lemma.
\end{proof}

 Equation \eqref{eqn:n1.4} suggests that to show $u$ is Lipschitz, it suffices to prove that the integral
$$\int_X |\nabla_x G(x,y)|_{\omega_x} u(y) \omega^n(y)$$ is bounded independent of $x\in X^\circ$. However, in general this may not hold without additional conditions on $\omega$. Our strategy is to approximate this integral by the corresponding quantities with respect to $\omega_j$. More precisely, we choose cut-off functions $\rho_k$ as in Lemma \ref{cutoff} and approximate $u$ by letting 
$$u_k := u \rho_k.$$ 
$u_k$ can also seen as defined on $Y$ with compact support in  $\pi^{-1}(X^\circ)$.  We further define 
\begin{equation}\label{eqn:n1.6}
\hat v_{j,k}(x) = \int_Y G_j(x, y) u_k(y) \omega_j^n(y), 
\end{equation}
where $G_j$ is Green's function of $\omega_j$. %Note that $f_j$ depends on $a$, but we will suppress it for notational simplicity. 
$\hat v_{j,k}$ is uniformly bounded (independent of $j$ and $k$) in  $L^\infty(Y)$  as $u$ is bounded in $L^\infty(X)$. By \cite{GPSS1, GPSS2}, there exist $b>2n-2$ and $C_b>0$ such that 
$$
1\le G_j(x,y) \le \frac{C_b}{d_j(x,y)^{b}},~ {\rm for ~all~} x\neq y \in Y,
$$
where $d_j$ is the distance function induced by  $\omega_j$. Since we are interested in the gradient estimate of $\hat v_{j,k}$, we will  add a bounded constant $\Lambda>0$ to $\hat v_{j,k}$ so that  $v_{j,k} := \hat v_{j,k} + \Lambda$ satisfies
$$1\le v_{j,k} \le C $$
 for a uniform constant $C>0$ due to the uniform $L^\infty$ bound for $u_k$. It is clear that for each $j$ we have
\begin{equation}\label{eqnvjk}
 \Delta_{\omega_j} v_{j,k} = - u_k.
\end{equation}
 Hence $|\Delta_{\omega_j} v_{j,k}| \le C$ for all $j$ and $k$. Integrating the equation $v_{j,k} \Delta_{\omega_j} v_{j,k} = - u_k v_{j,k}$, we immediately obtain
\begin{equation}\label{eqn:f grad}
\int_Y \abs{\nabla v_{j,k}}_{\omega_j} \omega_j^n \le C.
\end{equation}
Moreover, we observe that from the local $C^2(Y\backslash \pi^{-1}(S))$ convergence of $\omega_j$ to $\pi^*\omega$, elliptic regularity implies that by passing to a subsequence of $j$'s, we have 
\begin{equation}\label{eqn:f limit}v_{j,k} \xrightarrow{C^{3,\alpha}_{loc}(\pi^{-1}(X^\circ))} v_{\infty, k},\end{equation}
In particular, we have 
\begin{equation}\label{eqnvinfk1}
v_{\infty, k}(x) = \int_X G(x,y) u_k(y) \omega^n(y) + \Lambda,
\end{equation}
for any $x\in X\backslash S$, and on this set, 
\begin{equation}\label{eqnvinfk2}
\Delta_\omega v_{\infty,k} = - u_k.
\end{equation}
 Since $v_{\infty, k}$ is bounded, both equation (\ref{eqnvinfk1}) and (\ref{eqnvinfk2}) hold globally on $X$ and $v_{\infty, k} \in W^{1,2}(X, \omega)$. 
  %
%%%%%%%%%%%%%%
%%

Returning to $v_{j, k}$ and equation (\ref{eqnvjk}), we have the following lemma as an immediate corollary of Theorem \ref{linest1}. 
\begin{lemma}\label{lem 3-4}
There exists  $C>0$    such that for all $j$ and $k$, we have  
\begin{equation}\label{eqn:s33}
\int_Y \left(\abs{\nabla \nabla v_{j,k}}_{\omega_j} + \abs{\nabla \overline{\nabla} v_{j,k}}_{\omega_j} ) +   |\nabla v_{j,k}|_{\omega_j}^4 \right) \omega_j^n  \le C.
\end{equation}
\end{lemma}

\begin{corollary}\label{cor 1.1}
There exists   $C>0$ such that 
\begin{equation}\label{eqn:u bound}
\int_X \left(\abs{\nabla \nabla u}_\omega + \abs{\nabla \overline{\nabla} u}_\omega +   |\nabla u|_\omega^4 \right) \omega^n  \le C.
\end{equation}
\end{corollary}
\begin{proof}
Passing to the limit as in \eqref{eqn:f limit}, we can apply Fatou's Lemma and Lemma \ref{lem 3-4} to obtain 
\begin{equation}\label{eqn:n1.15}
\int_X \left(\abs{\nabla \nabla v_{\infty, k}}_\omega + \abs{\nabla \overline{\nabla} v_{\infty, k}}_\omega  +   |\nabla v_{\infty, k}|_\omega^4 \right) \omega^n  \le C.
\end{equation}
Recall that as $k\to \infty$ we have 
$$v_{\infty, k} (x) = \int_{X} G(x,y) u_k(y) \omega^n(y) + \Lambda \to \int_{X} G(x,y) u(y) \omega^n(y) + \Lambda = \lambda^{-1} u(x) + \Lambda,$$
and this convergence is locally uniform in $C^3(K, \omega)$ for any compact subset $K\subset X^\circ$ by elliptic regularity of equation (\ref{eqnvinfk2}). Again by Fatou's Lemma, we conclude from \eqref{eqn:n1.15} that 
\begin{equation*}%\label{eqn:n1.15}
\int_X \left(\abs{\nabla \nabla u}  + \abs{\nabla \overline{\nabla} u}   +   |\nabla u|^4\right)\omega^n  \le C.
\end{equation*}
This proves the corollary. 
\end{proof}

We now   fix a point $x\in \{\rho_k = 1\}^\circ$ and identify $x$ with its pre-image $\pi^{-1}(x)\in Y$. We can also identify $\pi^* u_k$ and $u_k$ since $u_k$ is compactly supported on $X^\circ$.  We denote
$$U_k = X \setminus \{\rho_k=1\}.$$
For convenience, we can also identify $U_k$ as an open subset in $Y$.

\begin{lemma}\label{lem 3-6}
For any $k_0>0$ and $x\in \{ \rho_{k_0} = 1\}$, there exists  $C=C(k_0, x)>0$   such that for all $k>k_0$ 
\begin{equation*} 
\int_{U_k} \left( \abs{\nabla \nabla G(x,\cdot)}_\omega + \abs{\nabla \nabla G(x,\cdot)}_\omega + |\nabla G(x,\cdot)|_\omega^4  \right) \omega^n \le C.
\end{equation*}
\end{lemma}

\begin{proof}  Since $x$ is disjoint from $U_k$ for all sufficiently large $k$, there exists $c>0$ such that for all sufficiently large $i, k$, $d_j(x, \partial U_k)> c$. The lemma then follows from Proposition \ref{prop 2-3}. 
 \end{proof}

\begin{lemma}\label{lem 3-8}
There exists $C>0$ such that for any $x\in X^\circ$, we have
$$|\nabla u|_\omega(x) \le C + (\lambda+ 1) \int_X G(x,\cdot) |\nabla u|_\omega \omega^n.$$
\end{lemma}
\begin{proof}
By Bochner's formula for $\omega$ on $X^\circ$, we have  
$$\Delta_\omega \abs{\nabla u} \ge  \abs{\nabla \nabla u}  + \abs{\nabla \overline{\nabla} u}  - 2\lambda \abs{\nabla u} - 2  \abs{\nabla u}. $$
Applying  Kato's inequality, we have  
\begin{equation}\label{eqn:new1.21}
\Delta  |{\nabla u}|  \ge  - (\lambda+1) |{\nabla u}| .
\end{equation}
We now apply Green's formula to the smooth function with compact support, $\rho_k |\nabla u|_\omega$. Then for any $x\in X^\circ$, 
\begin{align*}
&~ |\nabla u|(x)=  |\nabla u|(x)\rho_k(x)\\
\le &~ \frac{1}{V} \int_X \rho_k |\nabla u| \omega^n + \int_X G(x,\cdot)\big( |\nabla u| \Delta  \rho_k + 2 \innpro{\nabla \rho_k, \nabla |\nabla u|}_\omega + (\lambda+1) \rho_k |\nabla u|   \big)\omega^n\\
=& I_1 + I_2
\end{align*}
for sufficiently large $k$.
We will estimate the three terms in $I_2$. 

In a fixed small neighborhood $U_x\subset X^\circ$ of $x$, $|\nabla u|$ is smooth and bound, while $G(x,\cdot)$ is integrable in $U_x$. In  $X\setminus U_x$, $G(x,\cdot)$ is bounded by a constant (possibly dependent on $x$) and $|\nabla u|$ is integrable. Hence the function $G(x,\cdot)|\nabla u|$ is $L^1$ integrable in $(X,\omega^n)$. Then by the dominated convergence theorem, we have for the third term
$$\lim_{k\to \infty}\int_{X} (\lambda+1)G(x, \cdot) |\nabla u| \rho_k \omega^n =(\lambda+1)  \int_{X} G(x, \cdot) |\nabla u| \omega^n. $$ 

For the second term, we observe that the corresponding integral is in fact taken over $U_k$, where $G(x,\cdot)$ is bounded by a constant $C_x>0$. By H\"older's inequality, we have
\begin{align*}
\Big|\int_X G(x,\cdot) 2 \innpro{\nabla \rho_k, \nabla |\nabla u|}_\omega \omega^n \Big|\le &~ C_x \int_{U_k} |\nabla \rho_k| \cdot |\nabla |\nabla u\|\omega^n\\
\le &~ C_x \Big( \int_{U_k} |\nabla \rho_k|^2\omega^n \Big)^{1/2} \cdot\Big( \int_{U_k} |\nabla |\nabla u\|^2\omega^n\Big)^{1/2}\\
\le &~  C_x \Big( \int_{U_k} |\nabla \rho_k|^2\omega^n \Big)^{1/2} \to 0
\end{align*}
 as $ k\to \infty$, where  Kato's inequality and Corollary \ref{cor 1.1} are used in the last inequality. 
 
 For the first term of $I_2$, we rewrite it as 
\begin{align}\label{eqn:new1.22}
\int_X G(x,\cdot) |\nabla u| \Delta_\omega \rho_k = &~  - \int_X \big( |\nabla u| \innpro{\nabla G(x,\cdot), \nabla \rho_k} + G(x,\cdot) \innpro{\nabla \rho_k, \nabla |\nabla u|}_\omega\big).
\end{align}
The second term on the RHS of \eqref{eqn:new1.22} tends to zero by similar argument before. 
The first term on the RHS of \eqref{eqn:new1.22} satisfies 
\begin{align*}
\Big|  \int_X  |\nabla u| \innpro{\nabla G(x,\cdot), \nabla \rho_k}  \Big| \le &~ \Big(\int_{U_k} |\nabla u|^2 |\nabla G(x,\cdot)|^2 \omega^n \Big)^{1/2} \| \nabla \rho_k\|_{L^2}\\
\le &~ \Big(\int_{U_k} (|\nabla u|^4 +  |\nabla G(x,\cdot)|^4) \omega^n \Big)^{1/2} \| \nabla \rho_k\|_{L^2}
\\
\le &~ C_x \| \nabla \rho_k\|_{L^2} \to 0, \text{ as }a\to\infty,
\end{align*}
where we apply Corollary \ref{cor 1.1} and Proposition \ref{prop 2-3}. Combining the above estimates, we obtain that 
\begin{equation}\label{eqn:new1.23}|\nabla u|_\omega(x) \le C + (\lambda+ 1) \int_X G(x,\cdot) |\nabla u|_\omega \omega^n,
\end{equation}
for some constant $C = \frac{1}{V_\omega} \int_X |\nabla u| \omega^n$ independent of $x$. This completes the proof of the lemma.
\end{proof}

We can now complete the proof of Proposition \ref{thm:lip2}.

\begin{proof}[Proof of Proposition \ref{thm:lip2}]

We fix $\gamma =2N+1$ for some sufficiently large integer $N$, so that  $\gamma^* = \frac{\gamma}{\gamma-1}= \frac{2N+1}{2N}$  is sufficiently close to $1$.  Then 
$\|G(x, \cdot)\|_{L^{\gamma^*}(X, \omega^n)}$ is uniformly bounded for all $x\in X$ by Lemma \ref{lemma green limit} (c.f. \cite{GPSS1, GPSS3}). 
We define a linear operator $\mathbf{T}$ by 
$$\mathbf{T} (f) (x) = \int_X G(x,y)f(y) \omega^n(y). $$
By H\"older's inequality, we have 
$$\| \mathbf{T}(f)\|_{L^\infty} \le \sup_{x\in X^\circ}\| G(x,\cdot)\|_{L^{\gamma^*}}\cdot   \| f\|_{L^\gamma}\le C \|f\|_{L^\gamma}.$$
On the other hand, by the generalized Minkowski inequality, we also have
$$
\| \mathbf{T}(f)\|_{L^{\gamma^*}} \le \int_X |f|(y) \Big(\int_X G(x,y)^{\gamma^*}\omega^n(x) \Big)^{1/\gamma^*} \omega^n(y)\le C \| f\|_{L^1}.
$$
Thus the linear operator $\mathbf{T}$ is of strong types $(\infty, \gamma)$ and $(\gamma^*, 1)$. Then by the Riesz–Thorin Interpolation Theorem, it follows that the operator $\mathbf T$ is also of strong type $(p(t), q(t))$, i.e., 
\begin{equation}\label{iter0}
\| \mathbf{T}(f)\|_{L^{p(t)}} \le  C_t \| f\|_{L^{q(t)}}.
\end{equation}
for any $t\in (0,1)$, where 
\begin{equation}\label{eqn:old1.20}
\frac{1}{p(t)} = \frac{t}{\infty} + \frac{1 - t}{\gamma^*} =  \frac{1 - t}{\gamma^*}, \quad \frac{1}{q(t)} = \frac{t}{\gamma} + \frac{1-t}{1}. 
\end{equation}

Denote $\phi = |{\nabla u}|$.  Set  $t= \gamma^*/2$ for \eqref{eqn:old1.20}. We have
$$p'_1 = \frac{2\gamma^*}{2-\gamma^*} = \frac{2\gamma}{\gamma -2}>2, ~q'_1=2.$$
Then  
\begin{equation}\label{iter1}
\| \mathbf{T}( \phi)\|_{L^{p'_1}} \le A_1 \| \phi\|_{L^2} \le B_1
\end{equation}
for some fixed constants $A_1, B_1>0$, since the $L^2$ norm of $\phi$ is bounded by \cite{GPSS2}. 
Lemma \ref{lem 3-8} implies that 
\begin{equation}\label{iter2}
\phi(x) \le C + (1+\lambda) \mathbf{T}(\phi)(x).
\end{equation}
 Hence $\| \phi\|_{L^{p'_1}} $ is also bounded. We can now iterate the above process by repeatedly applying (\ref{iter0}) and (\ref{iter2}).  Applying (\ref{eqn:old1.20}) inductively for
 $$ \gamma^* = p'_{k-1}, ~ t = \gamma^*/2, $$
  we have 
 $$ \frac{1}{p'_{k}} = \frac{1}{p'_{k-1}} - \frac{1}{\gamma},$$
or equivalently, 
$$p'_{k} = \frac{ 2 \gamma}{\gamma-2k}$$
whenever $k\leq N = (\gamma-1)/2$. Therefore there exists $B_k>0$ such that 
\begin{equation}\label{eqn:key 4}
\| \mathbf{T}(\phi)\|_{L^{p'_{k}}} \le A_k \| \phi\|_{L^{p'_{k-1}}} \le B_k
\end{equation}
for some uniform constants $B_k\geq 0$.
In particular,   we have 
$$p'_N =2\gamma >\gamma$$
and  so
$$\| \mathbf{T}(\phi)\|_{L^\gamma} \le C$$
from \eqref{eqn:key 4}.  
Applying Lemma \ref{lem 3-8} again, this implies that $\|\phi\|_{L^\gamma}\le C$. Together with this and H\"older's inequality we obtain for any $x\in X^\circ$
$$\phi(x) \le C + C \| G(x,\cdot)\|_{L^{\gamma^*}} \| \phi\|_{L^\gamma} \le C,$$ 
which gives the desired $L^\infty$-estimate for $\phi = |{\nabla u}|$.
\end{proof}
\begin{remark}
One can also show the $L^\infty$-bound of $|\nabla u|_\omega$ under a slightly weaker condition on the $L^{2-\epsilon}$-norm of  $\ric(\omega_j)_-$ for a small $\epsilon>0$. The proof is analogous to that in Proposition \ref{thm:lip2} by invoking the inequality \eqref{eqn:weak1} in Remark \ref{remark 2.1} and the following improved Kato's inequality: for some $\alpha\in (0,\frac 1 2)$ and $C>0$,
$$\Delta_\omega (1 + \abs{\nabla u}_\omega)^\alpha\ge - C (1 + \abs{\nabla u}_\omega)^\alpha,$$
which can be proved by using the fact that $\Delta_\omega u$ is bounded. 
\end{remark}

\begin{corollary} \label{rcdsp} $(\hat X, d_\omega, \omega^n)$ is a compact ${\rm RCD}(-1, 2n)$ space.

\end{corollary}

\begin{proof} Since $\omega$ has bounded local potentials, the singular set $\mathcal{S}(X)= X\setminus X^\circ$ as an analytic subvariety of $X$ must have capacity $0$. Since $\omega$ is smooth on $X^\circ$ and $\ric(\omega)$ is bounded below by $-1$ on $X^\circ$, $(\hat X, d_\omega, \omega^n)$ is a compact almost smooth metric measure space introduced in \cite{Ho} (c.f. Defintion 3.1 \cite{Ho}) by the diameter and spectral theory established in \cite{GPSS2}.  One can apply the criterion (Corollary 3.10 \cite{Ho} and Corollary 8 \cite{Sz24}) on the RCD propery for almost smooth metric measure spaces. Except the Lipschitz condition for eigenfunctions, all other conditions in the criterion ((1), (2), (4) in Corollary 3.10 \cite{Ho}) are satisfied for all singular K\"ahler metrics in $ \mathcal{V}(X, \theta_X, n, A, p, K)$ with Ricci curvature bounded below on $X^\circ$ by the diameter bound and the spectral theory obtained in \cite{GPSS2}. The remaining Lipschitz condition for eigenfunctions for $\omega$ can now be reduced to Proposition \ref{thm:lip2}. Therefore $(\hat X, d_\omega, \omega^n)$ is indeed a non-collapsed ${\rm RCD}(-1, 2n)$ space. 
\end{proof}

 %%%%%%%%%%%%%%%%%%%%%%%%%%%%%%%%%

\section{Schwarz lemma}

In this section, we will prove the following Schwarz lemma.

\begin{prop} \label{schw1} Let $X$ be an $n$-dimensional normal K\"ahler variety.   Suppose that $\omega \in \mathcal{V}(X, \theta_X, n, A, p, K)$ is a singular K\"ahler metric that admits a regularization $(Y, \{\omega_j\}_{j=1}^\infty)$ satisfying \eqref{cls2},  \eqref{nashb2} and \eqref{ricp2}. If 
$$\ric(\omega) \geq - \omega$$ 
on $X^\circ$, 
then there exists $c=c(X, \theta_X, n, A, p, K)>0$ such that on $X^\circ$, we have 
$$\omega \geq c\theta_X. $$

\end{prop}

\begin{proof}    We will choose a smooth K\"ahler form $\theta_Y \in [\alpha]$ on $Y$ and a smooth metric $\omega_0\in [\omega]$ on $X$ and  let 
$$\theta_j = \pi^*\omega_0+ \epsilon_j \theta_Y \in [\omega_j].$$
Then $\theta_j > \pi^*\theta_X$ and we define $\varphi_j$ by
$$\omega_j = \theta_j + \ddbar \varphi_j, ~\sup_X \varphi_j =0.$$
By the uniform Nash entropy bound for $\omega_j$, there exists $C>0$ such that
$$\|\varphi_j\|_{L^\infty(Y)} \leq C. $$
Furthemore, $\omega= \omega_0 + \ddbar \varphi$ for some $\varphi\in PSH(X, \omega_0)\cap L^\infty(X)$ with $\sup_X\varphi=0$. In particular, $\varphi_j \rightarrow \pi^*\varphi$ in $C^3_{loc}(\pi^{-1}(X^\circ))$ by our assumption for $\omega_j$.

Since the smooth K\"ahler metric $\theta_X$ is the restriction of smooth K\"ahler metric of ambient $\mathbb{C}^N$ for a local holomorphic embedding of $X$ into $\mathbb{C}^N$, the bisectional curvature of $\theta_X$  is uniformly bounded above. We can apply the Chern-Lu formula and obtain
\begin{equation}\label{eqn:chern}
\Delta_{\omega_j} \log \tr_{\omega_j} ( \pi^*\theta_X) \ge \frac{\innpro{\ric(\omega_j), \theta_X}_{\omega_j}}{\tr_{\omega_j}( \theta_X)} - B \tr_{\omega_j} (\theta_X),
\end{equation} 
for a constant $B\ge 0$ that depends on the upper bound of the bisectional curvature of $\theta_X$. For  $A>>1$ we have
\begin{equation}\label{eqn:n2.3}
\Delta_{\omega_j}   \log \xk {e^{-A \varphi_j} \tr_{\omega_j} (\pi^*\theta_X) } \ge - \ric(\omega_j)_- +  \tr_{\omega_j} (\pi^*(A\omega_0 - B\theta_X)) - An,
\end{equation} 

We let  $u_j: = e^{-A \varphi_j} \tr_{\omega_j} (\theta_X) $ and choose $A>>B+1$. Then from \eqref{eqn:n2.3} and there exists $C>0$ such that 
\begin{align*}\label{eqn:n2.4}
\Delta_{\omega_j} u_j \ge &~ - \ric(\omega_j)_- \cdot u_j + \frac{\abs{\nabla u_j}_j}{u_j} + 2 u_j^2 - C\\
\ge &~ - \left( \ric(\omega_j)_- \right)^2 + \frac{\abs{\nabla u_j}_j}{u_j} +  u_j^2 -C .
\end{align*} 
 Integrating this over $(Y, \omega_j^n)$, by our assumption on the $L^2$-bound for $\ric(\omega_j)_{-}$, there exists $C>0$ such that  
\begin{equation}\label{eqn:n2.4}
\int_Y \left(\frac{\abs{\nabla u_j}_j}{u_j}  + u_j^2   \right)\omega_j^n \le C.
\end{equation}
Let $u_\infty : = e^{-A\varphi}\tr_\omega (\theta_X)$. Since $\varphi_j$ and $\omega_j$ converge smoothly to $\pi^*\varphi$ and $\pi^*\omega$ on $\pi^{-1}(X^\circ)$ as  $i\to\infty$, respectively,  we have 
\begin{equation}\label{eqn:n2.5}
\int_X \bk{\frac{\abs{\nabla u_\infty}_\omega}{u_\infty+1}  + u_\infty^2  }\omega^n \le C.
\end{equation}
We consider the function $\phi = \log (u_\infty + 1)$. By Chern-Lu inequality again and by choosing $A>>1$, there exists $C>0$ such that on $X^\circ$, 
\begin{align*}
\Delta_\omega \phi = &~ \frac{u_\infty  \Delta_\omega \log u_\infty}{u_\infty  + 1 } + \frac{\abs{\nabla u_\infty}_\omega}{u_\infty (u_\infty  + 1)}   - \frac{\abs{\nabla u_\infty}_\omega}{(u_\infty + 1)^2}\\
\ge &~ \frac{u_\infty (u_\infty- C)}{u_\infty + 1}  \ge  -C.
\end{align*}
We will use the same cut-off functions $\rho_k$ as constructed in Lemma \ref{cutoff}. Fix a point $x\in \{\rho_k = 1\}^\circ$. Since $\phi = \rho_k \phi$ on $ \{\rho_k=1 \}$, we have by Green's formula for the function $\phi \rho_k$ which has compact support in $X^\circ$
\begin{align}\label{eqn:n2.6}
\phi(x) = \frac{1}{V_\omega}\int_X \phi \rho_k \omega^n + \int_X G(x,\cdot) \bk{ - \phi \Delta \rho_k - 2\innpro{\nabla \rho_k, \nabla \phi}_\omega - \rho_k \Delta\phi   }\omega^n.
\end{align}
We estimate the terms on the RHS of \eqref{eqn:n2.6} separately.
\begin{eqnarray*}
&& \Big|\int_X G(x,\cdot) 2\innpro{\nabla \rho_k, \nabla \phi}_\omega \omega^n \Big| \\
&= &  \Big|\int_{X\setminus \{\rho_k=1\}} G(x,\cdot) 2\innpro{\nabla \rho_k, \nabla \phi}_\omega \omega^n \Big| \\
&\leq& C_x \bk{\int_X \frac{\abs{\nabla u_\infty}_\omega}{(u_\infty + 1)^2} \omega^n}^{1/2} \left( \int_X   | \nabla \rho_k |^2_\omega \omega^n \right)^{1/2}  \to 0,
\end{eqnarray*}
as $k \to\infty$, where we have used \eqref{eqn:n2.5} and the upper bound of $G(x,\cdot)$ on $U_k$. 
\begin{align*}
& \Big|\int_X G(x,\cdot) \phi \Delta_\omega\rho_k \omega^n \Big| \\
\le& \int_X |\nabla G(x,\cdot)| |\nabla\rho_k| \phi \omega^n + \int_X  G(x,\cdot) |\innpro{\nabla \rho_k, \nabla \phi}_\omega|   \omega^n \\
\le &  \bk{\int_{X\setminus \{\rho_k=1\}} \phi^2 \abs{\nabla G(x,\cdot)}_\omega \omega^n}^{1/2} \left( \int_X  | \nabla \rho_k |^2_\omega \omega^n \right)^{1/2}+  C_x \left( \int_X  | \nabla \rho_k |^2_\omega \omega^n \right)^{1/2}\\
\le & \left( \left( \int_{X\setminus \{\rho_k=1\}} \left( \phi^4 +  |\nabla G(x,\cdot)|^4_\omega\right) \omega^n \right)^{1/4} +  C_x  \right) \left( \int_X   | \nabla \rho_k |^2_\omega \omega^n \right)^{1/2}\\
\to& 0
\end{align*}
as $k\to \infty$, where we have used Proposition \ref{prop 2-3} and the elementary fact that 
$$\phi^4 = (\log(1+ u_\infty))^4\le C u^2_\infty$$ and then apply \eqref{eqn:n2.5}. These together with \eqref{eqn:n2.6} yields that $\phi(x)\le C$ for some constant $C$ independent of $x$. This gives the upper bound of $u_\infty = e^{-A \varphi} \tr_{\omega} (\theta_X)$  as well as $\tr_{\omega} (\theta_X)$. This completes the proof of the proposition. 
\end{proof}

%%%%%%%%%%%%%%%%%%%%%%%%%%%%%%%%%%

\section{Approximations by twisted cscK metrics}

Suppose $Y$ is an $n$-dimensional K\"ahler manifold. Let $\alpha$ be a smooth K\"ahler metric and $\beta$ be a smooth non-negative closed $(1,1)$-form. We define the twisted $J$-function by 
\begin{equation}\label{jfun}
\delta J_{\alpha, \beta}(\varphi) = \int_Y\delta \varphi (\beta - c \alpha_\varphi) \wedge \alpha_\varphi^{n-1}, ~ \alpha_\varphi = \alpha + \ddbar \varphi, 
\end{equation}
where $\varphi \in PSH(Y, \alpha) \cap L^\infty(Y)$ and $$c= c(\alpha, \beta) =  \frac{ \beta\cdot \alpha^{n-1}}{\alpha^n}.$$
Given a smooth volume form $\Omega$, the twisted Mabuchi energy is defined by
\begin{equation}\label{mabu}
M_{\alpha, \beta}(\varphi) = \int_Y \log \left(\frac{\alpha_\varphi^n}{\Omega} \right) \alpha_\varphi^n + J_{\alpha, \beta - \ric(\Omega)}(\varphi), 
\end{equation}
The Euler-Lagrange equation  will be a constant twisted scalar curvature equation 
\begin{equation}\label{tcsck}
R(\alpha_\varphi) - \tr_{\alpha_\varphi}(\beta) = constant. 
\end{equation}
%
%We let $\text{Aut}(Y, \beta)$ be the set of biholomorphic automorphisms fixing the form $\beta$.  
 There exists a constant twisted scalar curvature metric satisfying (\ref{tcsck}) on $Y$ if and only if the Mabuchi energy is proper by the works of Chen-Cheng \cite{CC}.

Let $X$ be an $n$-dimensional normal variety with log terminal singularities. Suppose $\eta$ is a smooth nonnegative closed $(1,1)$-form.  Now we will consider a twisted K\"ahler-Einstein metric  $\omega$ on $X$ satisfying  
\begin{equation}\label{ketap}
\ric(\omega) =  - \omega + \eta, 
\end{equation}
where $\eta$ is a smooth closed positive $(1,1)$-form on $X$. %
Equation (\ref{ketap}) implies that $[\omega]=[K_X] +[\eta]$ is a K\"ahler class. Let $\omega_X \in K_X + [\eta]$ be a smooth K\"ahler metric on $X$. We can choose a smooth adapted volume measure $\Omega$ on $X$ such that
$$\ric(\Omega) = - \ddbar \log\Omega = - \omega_X  + \eta. $$ 
Then the curvature equation (\ref{ketap}) is equivalent to the following complex Monge-Amp\`ere equation 
\begin{equation}\label{tkema}
(\omega_X + \ddbar \varphi)^n = e^{\varphi} \Omega. 
\end{equation}
Since $\varphi \in PSH(X, \omega_X)$ there exists $p>1$ such that 
$$\frac{\omega^n}{\Omega} \in L^p(X, \Omega).$$
Therefore $\varphi\in L^\infty(X) \cap C^\infty(X^\circ)$. In particular, $\omega$ is a singular twisted cscK metric satisfying
$$\mathrm{R}(\omega) - \tr_\omega(\eta) = n. $$
Since the sign of $\omega$ is negative, the Mabuchi energy $M_{\omega_X, \eta}$ is proper without assuming the triviality of the automorphism group (c.f. \cite{Sz24}).

Let $\pi: Y \rightarrow X$ be a resolution of singularities with $Y$ a smooth K\"ahler manifold. Let $\theta_Y$ be a smooth K\"ahler metric on $Y$. Our goal is to find a regularization satisfying the conditions in Definition \ref{calareg}  for $\omega$ on the nonsingular model $Y$.

\begin{prop} \label{regul}.  Suppose $-K_Y$ is $\pi$-nef. For any sufficiently small $\epsilon>0$, there exists a smooth twisted cscK metric $\omega_\epsilon \in [\pi^*\omega + \epsilon \theta_Y]$ satisfying 
\begin{equation}\label{szcsck}
\mathrm{R}(\omega_\epsilon)  - \tr_{\omega_\epsilon} (\pi^*\eta) = c_\epsilon. 
\end{equation}
Furthermore, there exist $\epsilon_j \rightarrow 0$ such that 
$\{ \omega_{\epsilon_j} \}$ is a regularization for $\omega$ as in Definition \ref{calareg}. In particular, $\omega_{\epsilon_j}$ converges smoothly to $\omega$ on any compact subset of $Y^\circ$.

\end{prop}

\begin{proof} One can apply the same argument in the proof of Theorem 3 in \cite{Sz24} (c.f. \cite{BJT}) to construct the twisted cscK metrics $\omega_\epsilon$ satisfying (\ref{szcsck}). The same argument show that there exist $p>n$ and $C>0$ such that for all $\epsilon$ sufficiently small, 
$$\int_Y \left| \log \frac{\omega_\epsilon^n}{\theta_Y^n} \right|^p \omega_\epsilon^n \leq C$$
with uniformly $C^k$-bounds for $\omega_\epsilon$  on any fixed compact subsets of $Y^\circ$. 

It suffices for us to verify (3) in Definition \ref{calareg}, i.e.  the Calabi energy is uniformly bounded for all $\omega_\epsilon$. 
If we let
$$\omega_\epsilon = \omega_X + \epsilon \theta_Y + \ddbar \varphi_\epsilon, ~\sup_X \varphi_\epsilon=0, $$
then there exists $C>0$ such that for all sufficiently small $\epsilon>0$, 
$$\sup_Y |\varphi_\epsilon| \leq C. $$
We let $$H_\epsilon =e^{- A\varphi_\epsilon} \tr_{\omega_\epsilon}(\theta_X) $$
for a fixed sufficiently large $A>0$.
Since $\theta_X$ is a smooth K\"ahler metric on $X$ via local holomorphic embeddings into Euclidean spaces, its curvature is uniformly bounded above. We can apply the standard Chern-Lu formula and there exists $C>0$ such that 
$$\Delta_{\omega_\epsilon} H_\epsilon  \ge ~ - |\ric(\omega_\epsilon)|  H_\epsilon  +  A^{2\over 3} H_\epsilon^2 -CA  \geq - A^{- \frac{1}{4}} |\ric(\omega_\epsilon)|^2 +  H_\epsilon^2 - CA .$$
Therefore 
$$\int_Y \left( \tr_{\omega_\epsilon}(\pi^*\theta_X) \right)^2 \omega_\epsilon^n  \leq A^{-\frac{1}{4}} \int_Y |\ric(\omega_\epsilon)|^2 \omega_\epsilon^n+ CA. $$
On the other hand, 
\begin{eqnarray*}
 \int_Y \mathrm{ R}(\omega_\epsilon)^2 \omega_\epsilon^n &=& \int_Y \left( \tr_{\omega_\epsilon}(\pi^* \eta) \right)^2 \omega_\epsilon^n + 2c_\epsilon \int_Y \mathrm{R}(\omega_\epsilon) \omega_\epsilon^n + c_\epsilon^2 \int_Y \omega_\epsilon^n\\
 &\leq & C \int_Y \left( \tr_{\omega_\epsilon}(\pi^*\theta_X) \right)^2 \omega_\epsilon^n + 2c_\epsilon \int_Y \mathrm{R}(\omega_\epsilon) \omega_\epsilon^n + c_\epsilon^2 \int_Y \omega_\epsilon^n\\
 &\leq & CA^{-\frac{1}{4}} \int_Y |\ric(\omega_\epsilon)|^2 \omega_\epsilon^n + CA \\
 &\leq & \frac{1}{2} \int_Y \mathrm{R}(\omega_\epsilon)^2 \omega_\epsilon^n + C+ CA
 \end{eqnarray*}
 if $A$ is sufficiently large.
Hence there exists $C>0$ such that for all sufficiently small $\epsilon>0$, we have
$$\int_Y \mathrm{R}(\omega_\epsilon)^2 \omega_\epsilon^n \leq C. $$
This completes the proof of the proposition.
\end{proof}

%%%%%%%%%%%%%%%%%%%%%%%%%%%%%%%%%%%%%%

 \section{Ricci curvature as a current}

In this section, we define the Ricci curvature on a normal K\"ahler space as a current as introduced in \cite{FGS2}. 
Let $X$ be an $n$-dimensional compact normal K\"ahler variety of log terminal singularities. There exists $m\in \mathbb{Z}^+$ such that $mK_X$ is a Cartier divisor on $X$.  We let $\Omega$ be a smooth adapted volume measure on $X$, i.e., $\Omega = f_U |\eta|^{2/m}$ on a local open set $U$ of $X$, where $\eta$ is a generator of  $K_X^m$ and $f_U$ is a nowhere vanishing smooth function on $U$. We let 
$$\ric(\Omega) = -\ddbar \log \Omega \in -[K_X].$$
Then $\ric(\Omega)$ is a smooth closed $(1, 1)$-form on $X$.

\begin{defn} \label{riclbdef} Let $X$ be a compact normal K\"ahler variety of log terminal singularities.  Suppose $\omega \in \mathcal{V}(X, \theta_X, n, A, p, K)$. The Ricci curvature of $\omega$ is said to be bounded below by $H\in \mathbb{R}$ if 
$$- \ddbar \log \frac{\omega^n}{\Omega} +\ric(\Omega) \geq H \omega$$
as currents, i.e., $$- \log\frac{\omega^n}{\Omega} \in PSH(X,  \ric(\Omega) -H\omega). $$
In particular, 
$$\ric(\omega) = - \ddbar\log \omega^n = - \ddbar \log \frac{\omega^n}{\Omega} +\ric(\Omega)$$ is defined to be the Ricci curvature of $\omega$ as a current. 

\end{defn}

 Similarly, we can define singular K\"ahler metrics with Ricci curvature bounded above by $H\in \mathbb{R}$ by requiring 
 $$\log \frac{\omega^n}{\Omega} \in PSH(X, H\omega - \ric(\Omega)).$$
 %%%%%%%%%%%%%%%%%%%%%%%%%%%%%%%%%%%%%%%%%%

 \section{Proof of Theorem \ref{thm:main3}: a special case}

In this section, we will prove a special case of Theorem \ref{thm:main3}.  Let us first recall the volume density for an RCD space. Let $(Z, d, \mu)$ be a non-collapsed ${\rm RCD}(K, N)$ space. For any $p\in Z$, we define the volume density at $p$ by
\begin{equation}
\nu_{Z}(p) = \lim_{r\rightarrow 0} \frac{ \vol B_Z(p, r)}{\vol B_{\mathbb{R}^N}(0,r)}.
\end{equation}

The following is the main result of this section.

\begin{theorem} \label{specthm3} Let $X$ be an $n$-dimensional normal K\"ahler variety with log terminal singularities. Suppose 
\begin{enumerate}

\item there exists a resolution $\pi: Y \rightarrow X$ such that  $K_Y$ is $\pi$-nef, 

\smallskip

\item $\omega\in H^{1,1}(X, \mathbb{R})\cap H^2(X, \mathbb{Q})$ is a singular K\"ahler metric in $\mathcal{V}(X, \theta_X, n, A, p, K)$, 

\smallskip
\item $\ric(\omega) +\omega$ is a non-negative smooth closed $(1,1)$-form on $X$. 
\end{enumerate}
Then  

\begin{enumerate}

\item the metric measure space $(\hat X, d_\omega, \omega^n)$ (as in \eqref{mmsp}) is a  compact RCD space homeomorphic to $X$, 

\smallskip

\item $\cS(\hat X) = \cS(X)$ has Hausdorff dimension no great than $2n-4$.

\smallskip

\item There exists $c>0$ such that 
\begin{equation}\label{schw7}
\omega\geq c \theta_X. 
\end{equation}
In particular, the identity map $\iota: (\hat X, d_\omega) \rightarrow (X, \theta_X)$ is Lipschitz.

\end{enumerate}

\end{theorem}

 \begin{proof} Let $\eta = \ric(\omega) + \omega$. Then by assumption $\eta\geq 0$ is a smooth closed $(1,1)$-form. By Proposition \ref{schw1}, $\omega$ dominates $\theta_X$, i.e.,  there exists $C>0$ such that 
 $$- \omega \leq \ric(\omega) \leq C \omega.$$ 
In particular,  the twisted K\"ahler-Einstein equation $ \omega$ satisfies the assumptions of Proposition \ref{regul}. Thefore there exist a regularization $(Y, \{\omega_ j \}_{j=1}^\infty)$ for $(X, \omega_i)$ satisfying \eqref{cls2},  \eqref{nashb2} and \eqref{ricp2}. 
Corollary \ref{rcdsp} implies that $\overline{(X^\circ, \omega)}$ is a compact non-collapsed ${\rm RCD}(-1, 2n)$ space with bounded Ricci curvature. 

It remains to prove that $\hat X$ is homeomorphic to the original variety $X$ with the dimension estimate for the singular set.  This can be proved by the same argument in \cite{Sz24} for K\"ahler-Einstein RCD spaces based on the partial $C^0$-estimates developed in \cite{T1, DS1, CDS2, LS}, because the Ricci curvature is uniformly bounded on $X^\circ$.   
 \end{proof}

%%%%%%%%%%%%%%%%%%%%%%%%%%%%%%%%%%%%%%%%%%%%%%

 \section{Proof of Theorem \ref{thm:main3}: the general case}

In this section, we will prove Theorem \ref{thm:main3} in full generality following the approximation approach in \cite{FGS2}. Suppose $\omega$ is a singular K\"ahler metric satisfying the assumptions of Theorem \ref{thm:main3}. 
We let $$\eta = \ric(\omega) + \omega \geq 0,$$ 
which is not necessarily smooth as considered in Theorem \ref{specthm3}. 
 If $\eta$ is not strictly positive,  we can let $\omega' = 2\omega $ and $\eta' = \omega + \eta$, then 
$\ric(\omega') = - \omega' + \eta' $ and $\eta'$ is indeed a K\"ahler current by Proposition \ref{schw1}. Therefore we can always assume $\eta$ is K\"ahler current that dominates the smooth K\"ahler metric $\theta_X$.

We will pick a smooth K\"ahler metric $\eta_0 \in [\eta]$ and let 
$$\eta=\eta_0 +\ddbar\psi, ~\psi \in PSH(X, (1+\varepsilon) \eta_0)\cap C^\infty(X^\circ).$$
 for some $\varepsilon>0$. 
The following regularization lemma is a special case proved in \cite{FGS2}.
\begin{lemma}  There exists a sequence of $\psi_j\in C^\infty(X) \cap PSH(X, \eta_0)$ such that $\psi_j \geq \psi$ and $\psi_j$ converges pointwise to $\psi$. Furthermore, $\psi_j$ converges to $\psi$ smoothly on any compact subset of $X^\circ$. 
\end{lemma}

\begin{proof} Since $[\eta] \in H^2(X, \mathbb{Q})$ is ample, we can choose $\eta_0$ to be the restriction of a smooth K\"ahler metric $\tilde \eta_0 $ on $\mathbb{CP}^N$ for some holomorphic embedding $\iota: X \rightarrow \mathbb{CP}^N$. Then for any $\eta_0$-PSH function on $X$, it can be extended to a $\tilde \eta_0$-PSH function on $\mathbb{CP}^N$ by \cite{CG}. Immediately, there exists $\psi_i \in PSH(X, \eta_0) \cap C^\infty(X)$ such that $\psi_i$ converge to $\psi$ decreasingly. Since $\psi$ is smooth on $X^\circ$, $\psi_i$ converges to $\psi$ in $L^\infty(K)$ for any $K\subset\subset X^\circ$. 

Let $L \rightarrow X$ be the ample $\mathbb{Q}$-line bundle with $\eta \in c_1(L)$. Let $\mathcal{J}$ be the ideal sheaf associated to the singular set $\cS(X)$. Then $mL \otimes \mathcal{J}$ is globally generated for sufficiently large $m\in \mathbb{Z}^+$. Let $h_0$ be the smooth hermitian metric for $L$ with $\ric(h_0) = \eta_0$ and $\sigma_0, ..., \sigma_{N_m}$ be a basis for $H^0(X, mL\otimes \mathcal{J})$. We let $\phi = \frac{1}{m} \log \left( \sum_{k=0}^{N_m} |\sigma_j|^2_{h_0}\right)$. Then $\phi \in PSH(X, \eta_0)\cap C^\infty(X^\circ)$ with $\phi$ tending to $-\infty$ near $\cS(X)$.  

Let $\tilde \psi_{i, \epsilon, \delta} = \mathcal{M}_{\epsilon} (\psi_i, \psi + \delta + \delta^2 \phi)$ be the regularized maximum of $\psi_i$ and $\psi+\delta + \delta^2 \phi$  for $\delta, \epsilon>0$. By definition, we have $\tilde\psi_{i, \epsilon, \delta}\in PSH(X, \eta_0)$ with $\tilde\psi_{i, \epsilon, \delta}\geq \psi_i\geq \psi$. Furthermore, $\tilde\psi_{i, \epsilon, \delta} \in C^\infty(X)$ since $\tilde \psi_{i, \epsilon, \delta} = \psi_i$ near $\cS(X)$ for sufficiently small $\epsilon>0$. 

For any $K\subset\subset X^\circ$,  $\tilde \psi_{i, \epsilon, \delta}= \psi  + \delta$ on $K$ for sufficiently large $i>0$ and sufficiently small $\delta>>\epsilon$, since $\psi_i$ converges to $\psi$ uniformly in $L^\infty(K)$.  At the same time, $\tilde \psi_{i, \epsilon, \delta} = \psi_i$ near $\cS(X)$ as $\psi+\delta + \delta^2 \phi$ tends to $-\infty$ along $\cS(X)$. By choosing suitable $\epsilon_i, \delta_i \rightarrow 0$, $\tilde \psi_{i, \epsilon_i, \delta_i}$ converges to $\psi$ smoothly on any fixed compact subset of $X^\circ$. This proves the lemma.
\end{proof}

We let $\eta_i = \eta_0+ \ddbar \psi_i$. Then $\alpha_i$ is sequence of smooth K\"ahler metrics on $X$. 
We can consider the twisted K\"ahler-Einstein equation
\begin{equation} \label{tkeqn}
\ric(\omega_i) = - \omega_i + \eta_i. 
\end{equation}
Let $\Omega$ be a smooth adaptive volume measure on $X$ such that 
$$\ric(\Omega) = -\omega_0 + \eta_0. $$
Then the   complex Monge-Amp\`ere equation equivalent to (\ref{tkeqn})  is given by 
\begin{equation}\label{tkema2}
(\omega_0 + \ddbar \varphi_i)^n = e^{\varphi_i  -\psi_i} \Omega.
\end{equation}
From the construction,  $-\psi_i \leq - \psi$  and $-\psi_i$ converges to $\psi$ pointwise. Since $\int_X e^{\varphi_i - \psi_i} \Omega = [\omega_0]^n$, $\varphi_i \in PSH(X, \omega_0)$ is uniformly bounded above.  Hence 
$$\omega_i = \omega_0 +\ddbar \varphi_i \in \mathcal{V}(X, \theta_X, n, A, p, K')$$
 for some  $K'>0$  for all $i>0$. Therefore $(\hat X_i, d_{\omega_i}, (\omega_i)^n)$ will satisfy the assumptions of Theorem \ref{specthm3}. In particular, $\hat X_i = X$.

\begin{lemma} There exists $C>0$ such that for all $i>0$, 
\begin{equation}\label{estsec81}
\|\varphi_i\|_{L^\infty(X)} \leq C, ~\mathrm{Diam}(\overline{(X^\circ, \omega_i)}) \leq C
\end{equation}
\begin{equation} \label{estsec82}
\omega_i \geq C^{-1} \theta_X.
\end{equation}

\end{lemma}
 
 \begin{proof} Since $\omega_i   \in \mathcal{V}(X, \theta_X, n, A, p, K')$, (\ref{estsec81}) follows immediately from \cite{GPSS1, GPSS2} due to the uniform Nash entropy bound. To achieve (\ref{estsec82}), we will apply the Chern-Lu argument one more time. For each $i>0$, there exists $c_i>0$ such that 
 $$\omega_i \geq c_i \theta_X$$
 by (\ref{schw7}) in Theorem \ref{specthm3}. 
 Let $D$ be an effective divisor on $X$ whose support contains the singular set of $X$. Let $\sigma_D$ be the defining section of $D$ and $h_D$ be a smooth hermitian metric for the line bundle associated to $D$. We consider 
 $$H_{i, \epsilon} = \log \tr_{\omega_i}(\theta_X) - A \varphi_i + \epsilon \log |\sigma_D|^2_{h_D}$$
 for small $\epsilon>0$. 
 $H_{i, \epsilon}$ must attain its maximum in $X\setminus D$. Straightforward calculations show that
 $$\Delta_{\omega_i} H_{i, \epsilon} \geq  \tr_{\omega_i} (\theta_X) - CA$$
in $X\setminus D$ for a fixed sufficiently large $A>0$ and a uniform constant $C$ that are both independent of $i$ and $\epsilon$ because $\ric(\omega_i) \geq - \omega_i$ for all $i$. One can apply the maximum principle to $H_{i, \epsilon}$ since $H_{i, \epsilon}$ is smooth on $X\setminus D$, which leads to the uniform upper bound independent on $i$ and $\epsilon$ for $H_{i, \epsilon}$ . The estimate (\ref{estsec82}) is then obtained after letting $\epsilon \rightarrow 0$ since $\varphi_i$ is uniformly bounded. 
 \end{proof}

\begin{corollary} \label{2ndord7} For any compact $K\subset\subset X^\circ$, there exists $C>0$ such that for all $i$, 
$$ C^{-1} \theta_X \leq \omega_i \leq C \theta_X. $$

\end{corollary}

\begin{proof} The corollary follows immediately from (\ref{estsec82}) and the complex Monge-Amp\`ere equation (\ref{tkema2}) as $\varphi_i - \psi_i$ is uniformly bounded on $K$.
\end{proof}

\begin{lemma} For any compact $K\subset\subset X^\circ$ and $k>0$, 
$$\|\varphi_i - \varphi\|_{C^k(K)} =0. $$

\end{lemma}

\begin{proof} Given $K\subset\subset X^\circ$, $\psi_i$ converges to $\psi$ smootly on $K$. Therefore $\varphi_i$ is uniformly bounded in $C^k(K)$ for any $k>0$ by Schauder estimates and elliptic regularity for linear equations with the second order estimate in Corollary \ref{2ndord7}. Since $\psi_i$ converges to $\psi$ smoothly in $K$ and point-wisely on $X$, we have 
$$\lim_{i\rightarrow \infty} \left\| e^{- \psi_i}  - e^{-\psi} \right\|_{L^1(X, \Omega)} = 0. $$
The stability theorem for complex Monge-Amp\`ere equation \cite{Ko2, DZ} implies $$\lim_{i \rightarrow \infty} \| \varphi_i - \varphi\|_{L^\infty(X)} = 0. $$ The lemma immediately follows. 
\end{proof}

As a consequence, $(X^\circ, \omega_i)$ converges smoothly to $(X^\circ, \omega)$.

Let $(\hat X_i, d_{\omega_i}, \omega_i^n)$ be the metric measure space uniquely associated to $(X, \omega_i, \omega_i^n)$. By Theorem \ref{specthm3}, $(\hat X_i, d_{\omega_i}, \omega_i^n)$ is a non-collapsed compact ${\rm RCD}(-1, 2n)$ space homeomorphic to $X$ for all $i$. Furthermore, $\cR(\hat X_i) = X^\circ$.

\begin{lemma} $(\hat X_i, d_{\omega_i}, \omega_i^n)$ converges to $(\hat X, d_\omega, \omega^n)$ as $i\rightarrow \infty$. In particular,  $(\hat X, d_{\omega}, \omega^n)$ is a non-collapsed compact ${\rm RCD}(-1, 2n)$ space. 

\end{lemma}

\begin{proof} We first note that $\omega_i$ converge smoothly to $\omega$ on $X^\circ$ and the diameter of $(\hat X_i, d_{\omega_i}, \omega_i^n)$ is uniformly bounded. Therefore after possibly passing to a subsequence, $(\hat X_i, d_{\omega_i}, \omega_i^n)$ as  non-collapsed RCD spaces converge to a non-collapsed compact ${\rm RCD}(-1, 2n)$ space $(Y, d_Y, \mu_Y)$ whose regular part contains $X^\circ$ with local isometry to $(X^\circ, \omega, \omega^n)$. By the convergence of measures, $Y\setminus X^\circ$ has $0$ measure. Furthermore, since $\omega_i$ has uniformly bounded local potentials and $X\setminus X^\circ$ is an analytic subvariety, we can conclude that $Y\setminus X^\circ$ has $0$ capacity. In particular,  $(Y, d_Y, \mu_Y)$ is an almost smooth metric measure space introduced by Honda \cite{Ho}.   Therefore $(Y, d_Y, \mu_Y) =\overline{(X^\circ, \omega, \omega^n)}$ and the lemma is proved.
\end{proof}

\begin{lemma} \label{dengap}There exists $\epsilon=\epsilon>0$ such that for any $p\in \hat X_i\setminus X^\circ$ and $i>0$, 
$$\nu_{\hat X_i}(p) < 1- \epsilon. $$
\end{lemma}

\begin{proof} We will prove by contradiction.  After rescaling, we can assume there exist $\varepsilon_i >0$ such that the Ricci curvature of $\omega_i' = (\varepsilon_i)^{-2} \omega_i$ is bounded by $-1$ and $1$. Suppose there exist $p_i \in \hat X_i\setminus X^\circ$ such that 
$\nu_{\hat X_i}(p_i) \rightarrow 1$. Then  $\nu_{\hat X_i}(p_i) \rightarrow 1$ as well. the same argument of dimension iteration in the proof of Theorem 36 in \cite{Sz24} will show that for sufficiently large $i$, the tangent cone at $p_i$ in $(X, \omega'_i)$ is the flat $\mathbb{R}^{2n}$, which would imply that $p_i$ is a regular point of $X$. Ccontradiction.

\end{proof}

We will now show that $\hat X$ as the metric completion of $(X^\circ, \omega)$ is homeomorphic to $X$. 

Let   $$\mathcal{R}_\epsilon (\hat X)= \{ p\in \hat X : \nu_{\hat X} (p) > 1- \epsilon\}. $$

\begin{lemma} There exists $\epsilon>0$ such that 
$$\mathcal{R}(\hat X) = \mathcal{R}_\epsilon (\hat X)=  X^\circ, $$

\end{lemma}

\begin{proof} For any $p\in \hat X \setminus X^\circ$, there exist $p_i \in \hat X_i \setminus X^\circ$ such that $p_i \in (\hat X_i, d_{\omega_i}) \rightarrow p$. By volume convergence and Lemma \ref{dengap},  $$\nu_{\hat X_i}(p) <1-\epsilon$$ for some uniform $\epsilon>0$.  This would imply that $p$ is not a regular point of $\hat X$ and so $\mathcal{R}(X) = X^\circ$. By Lemma \ref{dengap} again, $\mathcal{R}_\epsilon (\hat X)= \mathcal{R}(\hat X)$ if we choose sufficiently small $\epsilon>0$. The lemma then follows.
\end{proof}

\begin{lemma} \label{lsana} Let $(V, o)$ be any iterated tangent cone $(\hat X, p)$. Suppose the $\epsilon$-singular set $V\setminus \mathcal{R}_\epsilon(V)$ has $0$ capacity for some $\epsilon>0$.  There exist $M, C>0$ such that for some $m\leq M$, there exists $\sigma \in H^0(X, L^{m})$ satisfying 
\begin{enumerate}

\item $\|\sigma\|_{L^2(h^{m}, m\omega)} \leq C. $

\smallskip

\item $\left| |\sigma(z)| - e^{-md_{\hat X}(z, p)} \right| < \epsilon. $

\end{enumerate}

\end{lemma}

\begin{proof} The lemma is a consequence of Proposition 3.1 of \cite{LS} with slight modifications. If $\hat X$ is the Gromov-Hausdorff limit of uniformly non-collapsed polarized manifolds with uniform lower Ricci bound, the lemma would immediately follow from Proposition 3.1 of \cite{LS}. However, $\hat X$ is the limit of $(\hat X_i, d_{\omega_i})$ and one has to take care of the singularities of $\hat X$. 

Suppose $(V, o)$ is the pointed Gromov-Hausdorff limit of $(\hat X, p_i, A_i d_\omega)$ associated to the line bundle $L_i = A_i L$ with $A_i \rightarrow \infty$. Let 
$$\Sigma = \{ x\in V: ~ \nu_{(V,o)} (x) \leq 1- \epsilon \} $$
for a fixed sufficiently small $\epsilon>0$. 
Obviously, $\Sigma$ is closed by continuity of the tangent cones of a non-collapsed RCD space \cite{CN, Den}. Let $\Sigma_\delta$ be the $\delta$-neighborhood of $\Sigma$ and $U_{\delta, R}=B(o, R) \setminus \Sigma_\delta$. By choosing sufficiently small $\delta>0$, we can assume that  the open set $U_{\delta, R}$ does not contain any limit of singular points of $X$ by Lemma \ref{dengap},. 

In fact, $U_{\delta, R}$ is the Gromov-Hausdorff limit of a sequence of open smooth K\"ahler manifolds in $X^\circ$ with Ricci curvature  bounded below by $-1$. Let $U_i$ be lift of $U_{\delta, R}$ back in $X^\circ \subset \hat X_i$ under Gromov-Hausdorff approximation. One can follow exactly the same argument in the proof of Proposition 3.1 in \cite{LS} to construct holomorphic charts on $U_i$ and then apply the partial $C^0$-techniques with suitable cut-off functions to prove the lemma. 
\end{proof}

\begin{lemma} \label{split8} Any iterated tangent cone of $(\hat X, \omega)$ cannot split off $\mathbb{R}^{2n-1}$ or $\mathbb{R}^{2n-2}$. 

\end{lemma}

\begin{proof} The lemma follows from Proposition 10 and Proposition 27 of \cite{Sz24} along with Lemma \ref{lsana}. 
\end{proof}

The following lemma is a direct consequence of Lemma \ref{split8} and De Philippis-Gigli's dimension estimates \cite{DG2} for singular sets of non-collapsed RCD spaces.
 
\begin{lemma} The singular set of any iterated tangent cone of $\hat X$ has Hausdorff codimension at least $3$. In particular, it has $0$ capacity.

\end{lemma}

We remark that if $\omega$ has bounded Ricci curvature, then the Hausdorff codimension of the singular sets of any iterated tangent cone must be no less than 4.

\begin{lemma} \label{0cap8} There exists $\epsilon>0$ such that for any iterated tangent cone $(V, o)$ of $(\hat X, d_\omega, \omega^n)$,  $V\setminus \cR_\epsilon(V)$ has $0$ capacity. 
\end{lemma}

\begin{proof} We let $S_1 \subset \cS(V)$ be the set of limits of $\cS(\hat X)$. Then $S_1$ is closed and $S_1 \subset \cS_{n-3}(V)$. Therefore $S_1$ has capacity $0$. For any $\varepsilon>0$, $R>>1$ and any compact subset $K$ of $\cR_\epsilon (V)$, there exist a cut-off function $\rho_1$ such that $\rho_1=1$ on $K$, $\rho_1$ vanish in an open neighborhood $U_1$ of $S_1$ and
$$\|\nabla \rho_1\|_{L^2(B(o, R))} < \varepsilon^2. $$
Since $V\setminus U_1$ is the limit of regular points in $X^\circ$, it is the Gromov-Hausdorff limits of non-collapsed open K\"ahler manifolds with Ricci curvature uniformly bounded below. Then the results of \cite{LS} implies that $S_2 = (V\setminus U_1) \setminus \cR_\epsilon (V)$ has $0$ capacity by \cite{LS}. In fact, $S_2$ is contained in an analytic subvariety of $V\setminus U_1$.  By the choice of $U_1$, $K \subset (V\setminus U_1) \setminus \cR_\epsilon(V)$.   Therefore there exists a cut-off function $\rho_2$ with support in $V\setminus U_1$ such that $\rho_2=1$ on $K$, $\rho_1$ vanish in an open neighborhood of $ (V\setminus U_1) \setminus \cR_\epsilon (V)$ and 
$$\|\nabla \rho_2\|_{L^2(B(o, R)\setminus U_1)} < \varepsilon^2. $$
Then $\rho=\rho_1\rho_2$ is the cut-off function such that $\rho=1$ on $K$, $\rho$ vanish in an open neighborhood of $V \setminus \cR_\epsilon (V)$ and 
$$\|\nabla \rho\|_{L^2(B(o, R))} < \varepsilon. $$
This completes the proof of the lemma.
\end{proof}

Immediately, we can strengthen Lemma \ref{lsana}.

\begin{corollary}\label{lsana88} Let $(V, o)$ be any iterated tangent cone $(\hat X, p)$. There exist $M, C>0$ such that for any $\epsilon>0$, there exist some $m\leq M$ and $\sigma \in H^0(X, L^m)$ satisfying 
\begin{enumerate}

\item $\|\sigma\|_{L^2(h^m, m\omega)} \leq C. $

\smallskip

\item $\left| |\sigma(z)| - e^{-md_{\hat X}(z, p)} \right| < \epsilon. $

\end{enumerate}

\end{corollary}

We can now complete the proof of Theorem \ref{thm:main3} by applying Corollary \ref{lsana88} to separate points as in \cite{DS1, S2}. 

%%%%%%%%%%%%%%%%%%%%%%%%%%%%%%%%%

 \bigskip
 \bigskip
 
\noindent {\bf{Acknowledgements:}} The authors would like to thank G\'abor Sz\'ekleyhidi and Freid Tong for a number of helpful discussions. They would also like to thank Yifan Chen, Shih-Kai Chiu, Max Hallgren, G\'abor Sz\'ekleyhidi and Tat Dat T\^{o} and Freid Tong for sending us their preprint \cite{CCHSTT} and sharing their ideas with us.


\begin{thebibliography}{99}


\bibitem{AGS} Ambrosio, L., Gigli, N. and Savare, G. {\em Bakry-\`Emery curvature-dimension condition and Riemannian Ricci curvature bounds},  Ann. Probab. 43 (2015), no. 1, 339--404

\bibitem{BBGZ} Berman, R., Boucksom, S., Guedj, V., and Zeriahi, A. {\em A variational approach to complex Monge-Amp\`ere equations}, Publ. Math. Inst. Hautes Etudes Sci. 117 (2013), 179--245

\bibitem{BJT} Boucksom, S., Jonsson, M. and Trusiani, A. {\em Weighted extremal K\"ahler metrics on resolutions of singularities}, arXiv:2412.06096





\bibitem{CC1}  Cheeger, J. and Colding, T.H. {\em On the structure of
space with Ricci curvature bounded below I}, J. Differential. Geom.
 {\bf 46} (1997), 406-480

\bibitem{CC2}  Cheeger, J. and Colding, T.H.  {\em On the structure of
space with Ricci curvature bounded below II}, J. Differential. Geom.
{\bf 52} (1999), 13--35

  \bibitem{CCT} Cheeger, J., Colding, T.H. and Tian, G.  {\em On the singularities
of spaces with bounded Ricci curvature}, Geom.Funct.Anal. Vol.12
(2002), 873--914


\bibitem{CC} Chen, X. and  Cheng, J. {\em On the constant scalar curvature K\"ahler metrics (II)—Existence results}, J. Amer. Math. Soc. 34 (2021), 937-1009


%\bibitem{CCHSTT} Chen, Y,., Chiu, S., Hallgren, M., Szekleyhidi, G., To, T. and Tong, F. {\em On K\"ahler-Einstein currents}, preprint



\bibitem{CDS1} Chen, X.X., Donaldson, S.K. and Sun, S. {\em K\"ahler-Einstein metrics on Fano manifolds. I: Approximation of metrics with cone singularities},  J. Amer. Math. Soc. 28 (2015), no. 1, 183--197

\bibitem{CDS2} Chen, X.X., Donaldson, S.K. and Sun, S. {\em K\"ahler-Einstein metrics on Fano manifolds. II: Limits with cone angle less than $2\pi$, } J. Amer. Math. Soc. 28 (2015), no. 1, 199--234

\bibitem{CDS3} Chen, X.X., Donaldson, S.K. and Sun, S. {\em K\"ahler-Einstein metrics on Fano manifolds, III: limits as cone angle approaches $2\pi$ and completion of the main proof}, J. Amer. Math. Soc. 28 (2015), no. 1, 235--278

\bibitem{CCHSTT} Chen, Y., Chiu, S.-K., Hallgren, M., S\'zekelyhidi, G., T\^{o}, T.D., and Tong, F. {\em On K\"ahler-Einstein currents}, preprint, 2025

\bibitem{CN} Colding, T.H. and Naber, A. {Sharp H\"older continuity of tangent cones for spaces with a lower Ricci curvature bound and applications},  Ann. of Math. (2) 176 (2012), no. 2, 1173--1229


\bibitem{CG} Coman, D., Guedj, V., and Zeriahi, A.  {\em On the extension of quasiplurisubharmonic functions},  Anal. Math. 48 (2022), no. 2, 411--426

\bibitem{Dav} Darvas, T. {\em Metric geometry of normal K\"ahler spaces, energy properness, and existence of
canonical metrics}, Int. Math. Res. Not. IMRN (2017), no. 22, 6752--6777

\bibitem{Den} Deng, Q. {\em Holder Continuity of Tangent Cones and Non-Branching in RCD(K,N) Spaces}, 
Thesis (Ph.D.)–University of Toronto (Canada), 2021

\bibitem{DG1} De Philippis, G. and Gigli, N. {\em From volume cone to metric cone in the nonsmooth setting}, Geom.
Funct. Anal. 26 (2016), no.6, 1526--1587

\bibitem{DG2} De Philippis, G., Gigli, N. {\em Non-collapsed spaces with Ricci curvature bounded from below}, J. Ec. polytech. Math. 5 (2018), 613--650

\bibitem{DZ} Dinew, S. and Zhang, Z. {\em On stability and continuity of bounded solutions of degenerate complex Monge-Amp\`ere equations over compact K\"ahler manifolds}, 
Adv. Math.   225 (2010), no. 1, 367--388

 \bibitem{DS1} Donaldson, S.K. and Sun, S. {\em Gromov-Hausdorff limits of K\"ahler manifolds and algebraic geometry},  Acta Math. 213 (2014), no. 1, 63--106


 \bibitem{EGZ} Eyssidieux, P.,   Guedj, V. and Zeriahi, A. {\em Singular
K\"{a}hler-Einstein metrics}, J. Amer. Math. Soc. 22 (2009), 607--639


\bibitem{FGS}  Fu, X., Guo, B. and  Song, J. {\em Geometric estimates for complex Monge-Amp\`ere equations}, J. Reine Angew. Math. 765 (2020), 69--99

\bibitem{FGS2} Fu, X., Guo, B. and  Song, J. {\em Singular K\"ahler spaces of complex dimension three with Ricci curvature bounded below}, preprint
 
  \bibitem{GT} Guedj, V. and T\^{o}, T.D. {\em K\"ahler families of Green's functions},  arXiv:2405.17232

 
 
\bibitem{GPSS1} Guo, B., Phong, D.H., Song, J.  and Sturm, J. {\em Sobolev inequalities on K\"ahler spaces}, preprint, 2023, arXiv:2311.00221

\bibitem{GPSS2}Guo, B., Phong, D.H., Song, J.  and Sturm, J.  {\em Diameter estimates in K\"ahler geometry},  {Comm. Pure Appl. Math.}, Volume 77, Issue 8 (2024), 3520-3556

\bibitem{GPSS3}Guo, B., Phong, D.H., Song, J.  and Sturm, J.   {\em Diameter estimates in K\"ahler geometry II: removing the small degeneracy assumption}, Math. Z. 308, 43 (2024).%accepted by {Math Z.}, 2024, arXiv:2405.18280

\bibitem{GPT} Guo, B., Phong, D.H. and  Tong, F.  {\em On $L^\infty$ estimates for complex Monge-Amp\`ere equations}, {Ann. of Math.} (2) 198 (2023), no.1, 393 - 418.


\bibitem{Ho}Honda, S. {\em Bakry-\`Emery conditions on almost smooth metric measure spaces}, Anal. Geom. Metr. Spaces 6 (2018), no. 1, 129 - 145.

\bibitem{LTW} Li, C., Tian, G. and Wang, F. {\em On the Yau-Tian-Donaldson conjecture for singular Fano varieties},
Comm. Pure Appl. Math. 74 (2021), no. 8, 1748--1800
 
\bibitem{LS} Liu, G. and Sz\'ekelyhidi, G. {\em Gromov-Hausdorff limits of K\"ahler manifolds with Ricci curvature bounded below}, Geom. Funct. Anal. 32 (2022), no. 2, 236--279

\bibitem{Ko1}  Ko\l{}odziej, S. {\em The complex Monge-Amp\`ere equation}, Acta Math. 180 (1998) 69--117

\bibitem{Ko2} Ko\l{}odziej, S. {\em The Monge-Amp\`ere equation on compact Kähler manifolds}, Indiana Univ. Math. J. 52 (2003), no. 3, 667--686

\bibitem{LV} Lott, J. and Villani, C. {\em Ricci curvature for metric-measure spaces via optimal transport} , Ann. of
Math. (2) 169 (2009), 903--991

\bibitem{PT} Pan, C. and To, T.D. {\em Weighted cscK metrics on K\"ahler varieties}, arXiv:2412.07968


 
\bibitem{S2} Song, J.  {\em Riemannian geometry of K\"ahler-Einstein currents}, arXiv:1404.0445

\bibitem{S3} Song, J. {\em Riemannian geometry of K\"ahler-Einstein currents II: an analytic proof of Kawamata's base point free theorem}, arXiv:1409.8374


\bibitem{S4} Song, J. {\em Degeneration of K\"ahler-Einstein manifolds of negative scalar curvature},  arXiv:1706.01518



\bibitem{ST1} Song, J. and Tian, G. {\em The K\"ahler-Ricci flow on surfaces of positive Kodaira dimension}, Invent. Math. {\bf 170} (2007), no. 3, 609--653

 \bibitem{ST2} Song, J. and Tian, G. {\em Canonical measures and K\"ahler-Ricci flow}, J . Amer. Math. Soc. 25 (2012), 303--353



\bibitem{St} Sturm, J. private notes, available at \url{https://sites.rutgers.edu/jacob-sturm/publications/}


\bibitem{Stk} Sturm, K. {\em  On the geometry of metric measure spaces. I.}, Acta Math. 196 (2006), no. 1, 65--131


\bibitem{Sz24}  Sz\'ekelyhidi, G. {\em Singular K\"ahler-Einstein metrics and RCD spaces}, 2024, arXiv:2408.10747

 \bibitem{T1} Tian, G. {\em On Calabi's conjecture for complex surfaces with positive first Chern class},  Invent. Math. 101 (1990), no. 1, 101--172

 \bibitem{T2} Tian, G. {\em  K\"ahler-Einstein metrics with positive scalar curvature}, Invent. Math. 130 (1997),
no.1, 1--37

 \bibitem{T3} Tian, G. {\em  Canonical metrics in K\"ahler geometry}, Lectures Math. ETH Zurich, Birkhauser Verlag, Basel, 2000. vi+101 pp.

\bibitem{T4} Tian, G. {\em K-stability and K\"ahler-Einstein metrics},  Comm. Pure Appl. Math. 68 (2015), no. 7, 1085--1156


\bibitem{Vu} Vu, D. { \em Uniform diameter estimates for K\"ahler metrics}, arXiv:2405.14680


   \bibitem{Y1}  Yau, S.-T. {\em On the Ricci curvature of a compact  K\"{a}hler
manifold and complex Monge-Amp\`{e}re  equation I}, \rm  Comm. Pure Appl.  Math.  31  (1978),  339--411

\bibitem{Zz} Zhang, Z. {\em On degenerate Monge-Amp\`ere equations over closed Kähler manifolds},  Int. Math. Res. Not. 2006, Art. ID 63640, 18 pp.

\bibitem{Zhe} Zheng, K. {\em Singular scalar curvature equations}, arXiv:2205.14726.


\end{thebibliography}
\end{document}